\documentclass[11pt]{amsart}
\usepackage[margin=1in]{geometry}

\usepackage{amssymb}
\usepackage{amsthm}
\usepackage{amsmath}
\usepackage{mathrsfs}
\usepackage{amsbsy}
\usepackage[all]{xy}
\usepackage{bm}
\usepackage{hyperref}
\usepackage{tikz}
\usepackage{array}
\usepackage{float}
\usepackage{enumerate}
\usepackage{xcolor}
\usepackage{hhline}
\allowdisplaybreaks
\usepackage{cite}

\newenvironment{enumerate*}%
  {\begin{enumerate}[(I)]%
    \setlength{\itemsep}{10pt}%
    \setlength{\parskip}{0pt}}%
  {\end{enumerate}}

\newtheorem{theorem}{Theorem}[section]
\newtheorem{proposition}[theorem]{Proposition}
\newtheorem{corollary}[theorem]{Corollary}

\newtheorem{question}[theorem]{Question}

\theoremstyle{definition}

\DeclareMathOperator{\DPT}{\mathsf{DPT}}
\DeclareMathOperator{\PT}{\mathsf{PT}}
\DeclareMathOperator{\con}{con}
\DeclareMathOperator{\non}{non}
\DeclareMathOperator{\ary}{-ary}

\begin{document}

\title[]{Supertrees} \keywords{}
\subjclass[2010]{}

\author[]{Colin Defant}
\address[]{Fine Hall, 304 Washington Rd., Princeton, NJ 08544}
\email{cdefant@princeton.edu}
\author[]{Noah Kravitz}
\address[]{Grace Hopper College, Yale University, New Haven, CT 06510, USA}
\email{noah.kravitz@yale.edu}
\author[]{Ashwin Sah}
\address[]{Massachusetts Institute of Technology, Cambridge, MA 02139 USA}
\email{asah@mit.edu}

\begin{abstract} 
A $k$-universal permutation, or $k$-superpermutation, is a permutation that contains all permutations of length $k$ as patterns.  The problem of finding the minimum length of a $k$-superpermutation has recently received significant attention in the field of permutation patterns.  One can ask analogous questions for other classes of objects.  In this paper, we study $k$-supertrees.  For each $d\geq 2$, we focus on two types of rooted plane trees called $d$-ary plane trees and $[d]$-trees.  Motivated by recent developments in the literature, we consider ``contiguous'' and ``noncontiguous'' notions of pattern containment for each type of tree.  We obtain both upper and lower bounds on the minimum possible size of a $k$-supertree in three cases; in the fourth, we determine the minimum size exactly.  One of our lower bounds makes use of a recent result of Albert, Engen, Pantone, and Vatter on $k$-universal layered permutations.
\end{abstract}
\maketitle

\section{Introduction}

\subsection{Background}\label{subsec:background}

Let $S_n$ denote the set of permutations of the set $[n]=\{1,\ldots,n\}$.  We write permutations as words in one-line notation. Given $\mu\in S_m$, we say that the permutation $\sigma=\sigma_1\cdots\sigma_n\in S_n$ \emph{contains the pattern} $\mu$ if there are indices $i_1<\cdots<i_m$ such that $\sigma_{i_1}\cdots\sigma_{i_m}$ has the same relative order as $\mu$. Otherwise, we say that $\sigma$ \emph{avoids} $\mu$. Consecutive pattern containment and avoidance are defined similarly by requiring the indices $i_1,\ldots,i_m$ to be consecutive integers. An enormous amount of research in the past half-century has focused on pattern containment and pattern avoidance in permutations \cite{Bona, Kitaev, Linton}. Plenty of particularly popular permutation pattern problems possess the following form: 
\begin{center}
What is the minimum length of a permutation that contains all patterns of a certain type?
\end{center}
For example, one can ask for the smallest size of a permutation containing all length-$k$ patterns; such a permutation is often called a \textit{$k$-universal permutation} or a \textit{$k$-superpermutation} \cite{Arratia, Eriksson, Miller}. The analogous question for consecutive pattern containment has also received attention \cite{Ashlock, Honner, Houston, Johnson}. Rather than discuss all of the variants of this problem that have emerged, we refer the reader to the beautiful article \cite{Engen}, which surveys many of the results in this area.

In recent years, the notion of pattern containment has spread to other combinatorial objects. It is natural to ask about the minimum possible sizes of ``universal objects'' in these contexts.  This idea dates back to 1964, when Rado \cite{Rado} asked for the minimum number of vertices in a graph that contains all $k$-vertex graphs as induced subgraphs. A vast amount of literature has been devoted to ``Rado's problem'' alone (see \cite{Alon, Alstrup, Butler, Chung, Esperet} and the references therein).

In this paper, we focus on rooted plane trees. Several variations on the theme of contiguous and noncontiguous pattern containment in rooted plane trees have appeared in \cite{Baril,Dairyko,Dotsenko,Flajolet1, Flajolet2, Gabriel,Pudwell, Rowland}. The purpose of the present article is to investigate the minimum possible size of a \emph{$k$-universal tree}, or \textit{$k$-supertree}, in some of these contexts.  Similar questions about universal trees have been studied since the 1960's \cite{CGSCaterpillars, ChungGraham1, ChungGraham2, ChungGraham3, ChungGraham4, Goldberg}. However, our notions of universal rooted plane trees are new and are inspired by more recent definitions of pattern containment in trees.

\subsection{Main Definitions and Terminology}\label{subsec:main-results}

Let $d\geq 2$ be an integer. A \emph{$d$-ary plane tree} is either an empty tree or a root vertex with $d$ subtrees that are linearly ordered from left to right and are themselves $d$-ary plane trees. A $2$-ary plane tree is also called a \emph{binary plane tree}. Note that the subtrees of a vertex can be empty. By the ``$i^\text{th}$ subtree" of a vertex, we simply mean the $i^\text{th}$ subtree from the left. We say an edge has ``type $i$" if it connects a vertex to the root of its $i^\text{th}$ subtree. A $d$-ary plane tree is called \emph{full} if every vertex has either $0$ or $d$ children (or, equivalently, if only leaves have empty subtrees).     

Every connected induced subgraph $T^*$ of a $d$-ary plane tree $\mathcal T$ can be viewed as a $d$-ary plane tree in the obvious way. If $T^*$ is isomorphic as a $d$-ary plane tree to another $d$-ary plane tree $T$, then we say that $T^*$ is a \emph{contiguous embedding} of $T$ in $\mathcal T$ and that $\mathcal T$ \emph{contiguously contains} $T$. For example, the $3$-ary plane tree
\[\begin{array}{l}
\includegraphics[height=2cm]{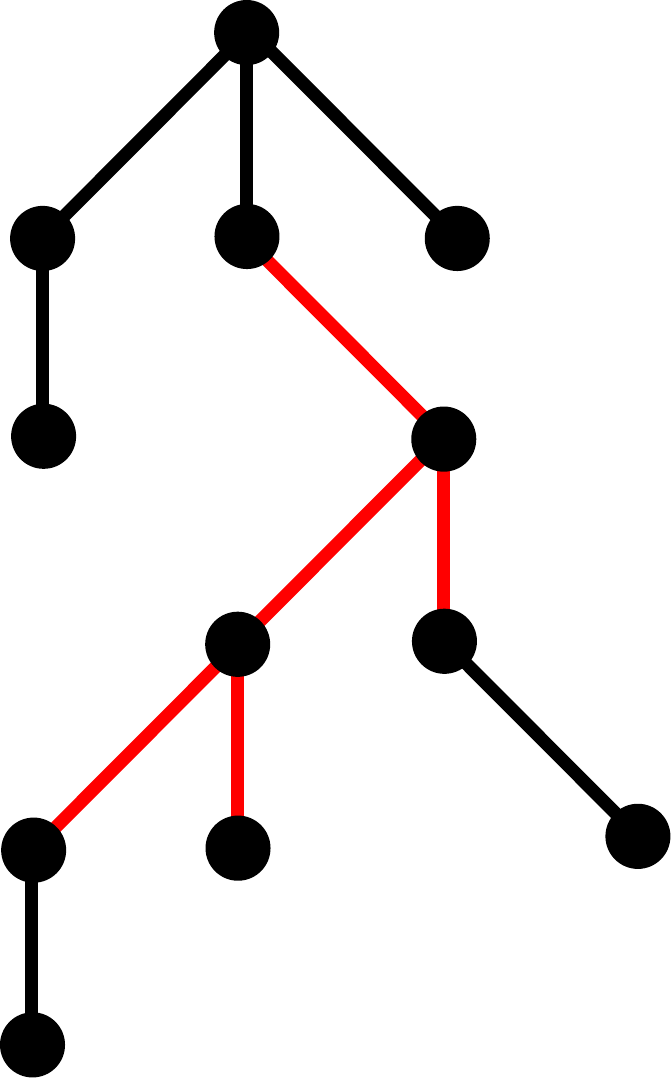}
\end{array}\text{ contiguously contains }\begin{array}{l}
\includegraphics[height=1.2cm]{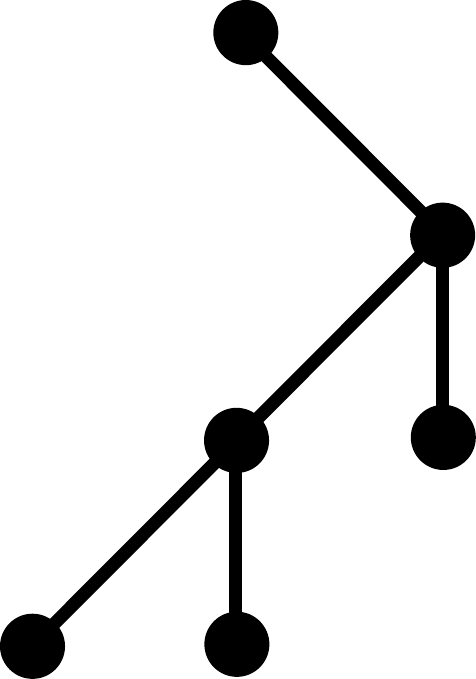}
\end{array}\text{ but does not contiguously contain }\begin{array}{l}
\includegraphics[height=1.2cm]{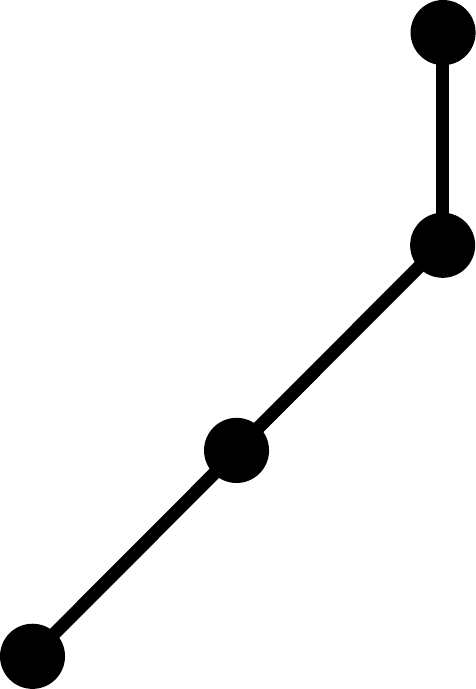}
\end{array}.\]
Rowland \cite{Rowland} defined contiguous pattern containment in full binary plane trees, and the authors of \cite{Gabriel} made a similar definition for full $3$-ary plane trees. In general, for any $k \geq 0$, the operation of removing (pruning) all leaves provides a natural bijection from the set of full $d$-ary plane trees with $dk+1$ vertices to the set of $d$-ary plane trees with $k$ vertices.  Using this bijection, one can easily see that our definition of contiguous pattern containment for $d$-ary plane trees corresponds to the definitions in \cite{Gabriel,Rowland} when $d\in\{2,3\}$. Our formulation has the advantage of working with smaller trees so that diagrams are not cluttered with unnecessary leaves.  

Given a vertex $u$ in a $d$-ary plane tree, let $\chi(u)$ be the set of all $i\in[d]$ such that $u$ has a nonempty $i^\text{th}$ subtree. Suppose $e$ is an edge of type $i$ that connects $u$ to one of its children $v$ (meaning $i\in\chi(u)$). We can consider the operation of \textit{contracting} the edge $e$. We call this operation a \emph{legal contraction} if every element of $\chi(u)\setminus\{i\}$ is either strictly smaller than $\min(\chi(v))$ or strictly greater than $\max(\chi(v))$. Informally speaking, this definition ensures that edges do not ``overlap" or ``cross'' each other during a legal contraction. After legally contracting an edge in a $d$-ary plane tree, we are left with a new $d$-ary plane tree. Given $d$-ary plane trees $\mathcal T$ and $T$, we say that $\mathcal T$ \emph{noncontiguously contains $T$} if we can obtain $T$ from $\mathcal T$ through a sequence of legal edge contractions. For example, the $3$-ary plane tree
\[\begin{array}{l}
\includegraphics[height=1.7cm]{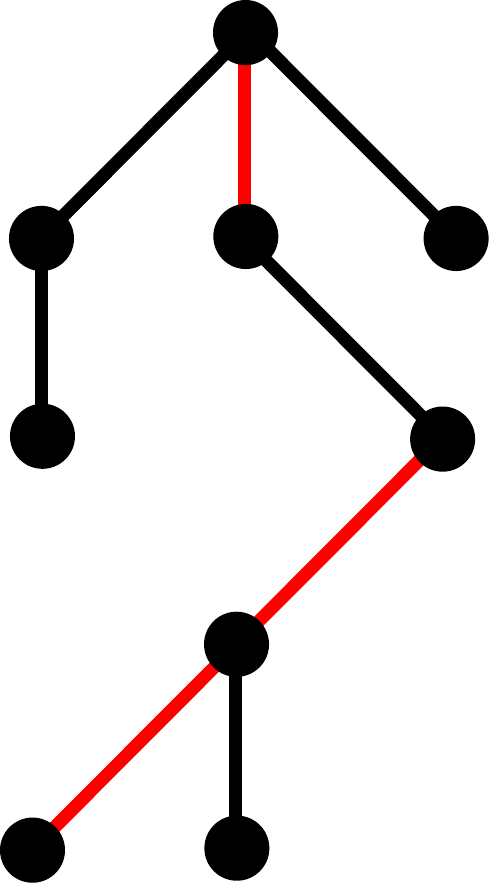}
\end{array}\text{ noncontiguously contains }\begin{array}{l}
\includegraphics[height=1.2cm]{SupertreesPIC6}
\end{array}\text{ but does not noncontiguously contain }\begin{array}{l}
\includegraphics[height=.85cm]{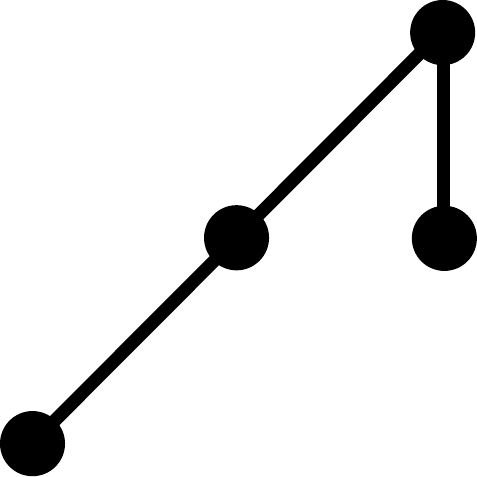}
\end{array}.\]
When $d=2$, we can use the pruning bijection mentioned above to show that our definition of noncontiguous pattern containment in binary plane trees is consistent with the notion considered in \cite{Dairyko,Pudwell}. 

Given a set $S$ of positive integers, an \emph{$S$-tree} is a rooted tree in which the children of each vertex are linearly ordered from left to right and the number of children of each vertex is an element of $S\cup\{0\}$. One can think of a $[d]$-tree as a tree obtained from a $d$-ary plane tree by forgetting about empty subtrees and the types of edges. What we term $[2]$-trees are more commonly called ``unary-binary trees" or ``Motzkin trees." 

Every connected induced subgraph $T^*$ of a $[d]$-tree $\mathcal T$ is itself a $[d]$-tree. If $T^*$ is isomorphic as a $[d]$-tree to another $[d]$-tree $T$, then we say that $T^*$ is a \emph{contiguous embedding} of $T$ in $\mathcal T$ and that $\mathcal T$ \emph{contiguously contains} $T$. For example, the $[3]$-tree 
\[\begin{array}{l}
\includegraphics[height=1.4cm]{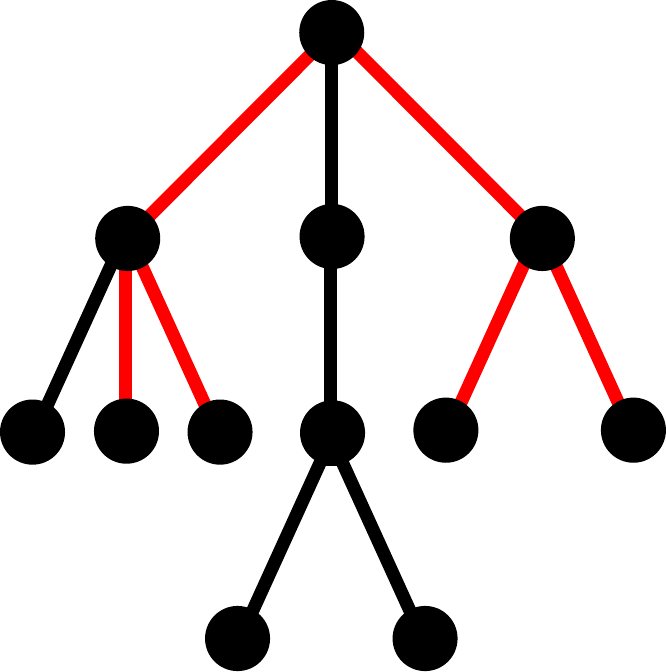}
\end{array}\text{ contiguously contains }\begin{array}{l}
\includegraphics[height=.95cm]{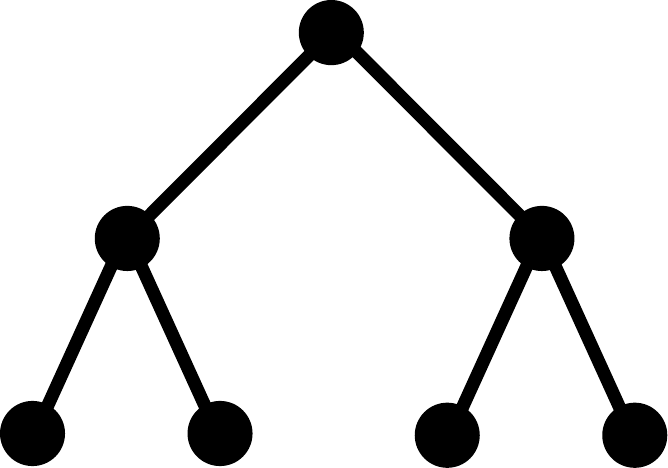}
\end{array}\text{ but does not contiguously contain }\begin{array}{l}
\includegraphics[height=1.2cm]{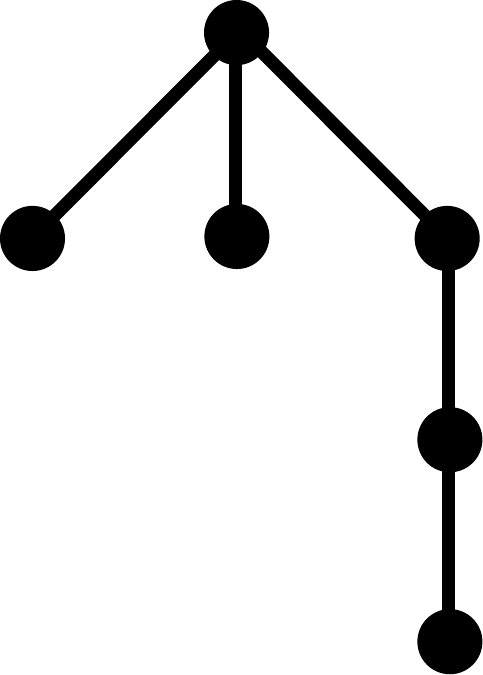}
\end{array}.\]

Suppose $e$ is an edge in a $[d]$-tree that connects a vertex $u$ to one of its children $v$. If the total number of children of $u$ and $v$, excluding $v$ itself, is at most $d$, then the operation of contracting the edge $e$ is a \emph{legal contraction}. Note that if $v'$ was a child of $u$ to the left (respectively, right) of $v$, then $v'$ remains to the left (respectively, right) of the children of $v$ after we legally contract $e$.  After legally contracting an edge in a $[d]$-tree, we are left with a new $[d]$-tree. Given $[d]$-trees $\mathcal T$ and $T$, we say that $\mathcal T$ \emph{noncontiguously contains $T$} if we can obtain $T$ from $\mathcal T$ through a sequence of legal edge contractions.  For example, the $[3]$-tree
\[\begin{array}{l}
\includegraphics[height=1.7cm]{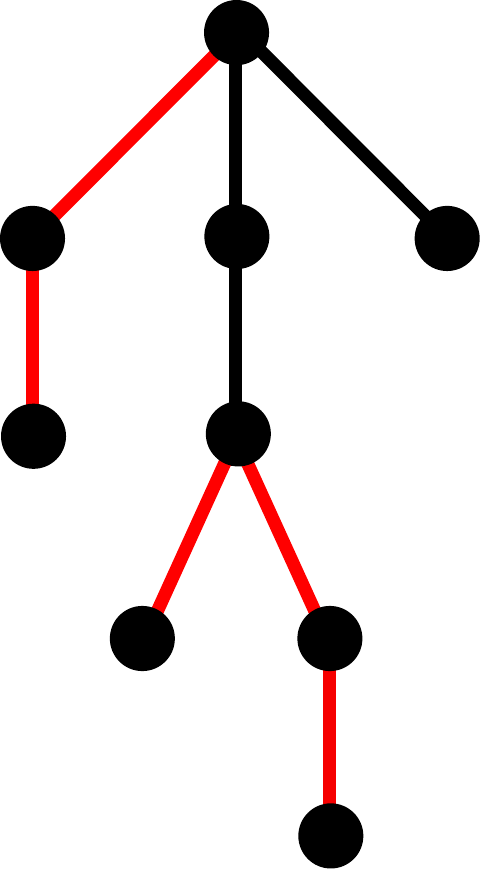}
\end{array}\text{ noncontiguously contains }\begin{array}{l}
\includegraphics[height=.95cm]{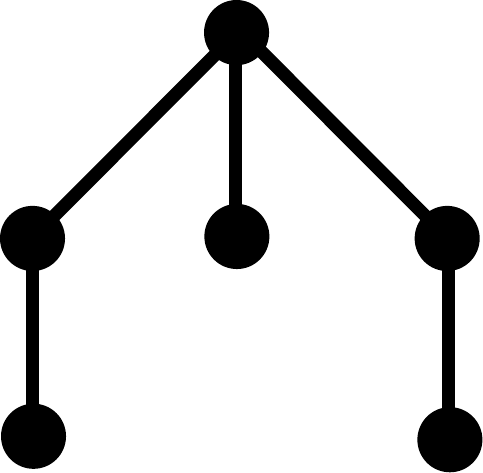}
\end{array}\text{ but does not noncontiguously contain }\begin{array}{l}
\includegraphics[height=1.2cm]{SupertreesPIC10}
\end{array}.\]

A \emph{contiguous $k$-universal $d$-ary plane tree} is a $d$-ary plane tree that contiguously contains all $d$-ary plane trees with $k$ vertices. Similarly, a \emph{noncontiguous $k$-universal $d$-ary plane tree} is a $d$-ary plane tree that noncontiguously contains all $d$-ary plane trees with $k$ vertices. Let $N_{d\text{-ary}}^\text{con}(k)$ (respectively, $N_{d\text{-ary}}^\text{non}(k)$) denote the minimum number of vertices in a contiguous (respectively, noncontiguous) $k$-universal $d$-ary plane tree. Contiguous and noncontiguous $k$-universal $[d]$-trees are defined analogously. Let $N_{[d]}^\text{con}(k)$ (respectively, $N_{[d]}^\text{non}(k)$) denote the minimum number of vertices in a contiguous (respectively, noncontiguous) $k$-universal $[d]$-tree. We refer to $k$-universal trees as ``$k$-supertrees" when the type of tree and the type of containment are clear from context. 

We say the root of a rooted plane tree has \emph{depth $0$}; a nonroot vertex has \emph{depth $r$} if its parent has depth $r-1$. The \emph{height} of a rooted plane tree is the maximum depth of its vertices.  We write $|T|$ for the number of vertices in $T$.  The \emph{perfect tree} $P_h^{(d)}$ is the unique $d$-ary plane tree of height $h$ that has exactly $d^r$ vertices of depth $r$ for each $r\in\{0,\ldots,h\}$.

It will be useful to have a formally defined ``gluing" operation for combining trees. Suppose $T$ is a rooted plane tree and $v$ is a leaf of $T$. If $T'$ is another rooted plane tree (of the same type as $T$, of course), then we can \textit{glue} $T'$ to $v$ by attaching $T'$ to $T$, where we identify the root of $T'$ with $v$. For example, if $T$ and $v$ are \[\begin{array}{l}
\includegraphics[height=1.6cm]{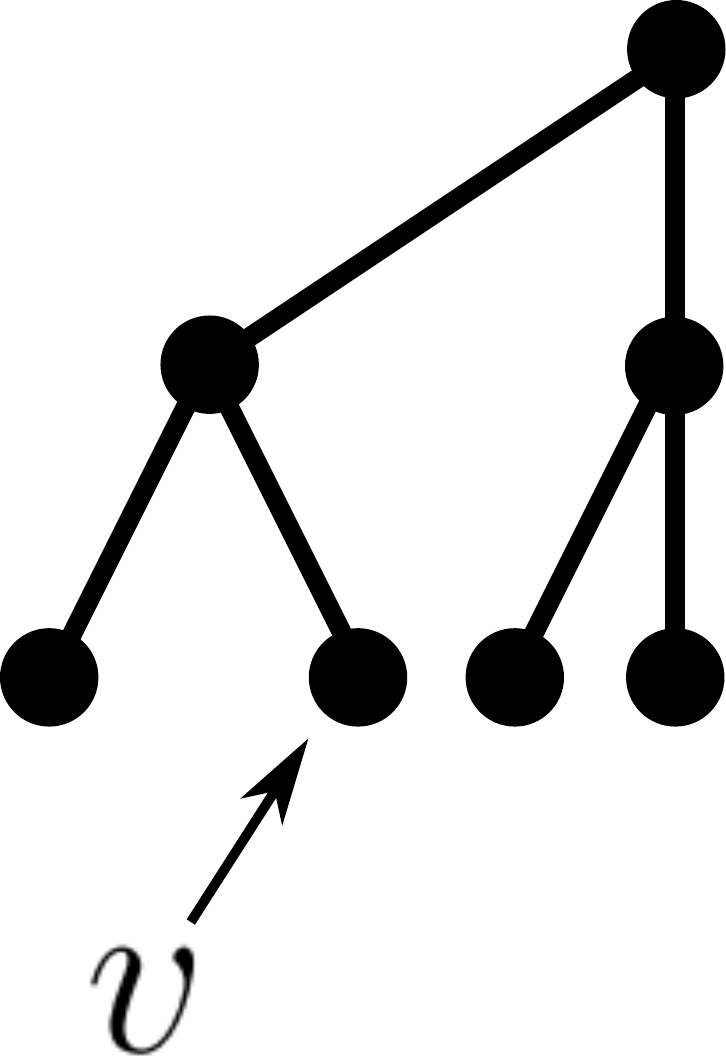}
\end{array}\text{ and }T'\text{ is }\begin{array}{l}
\includegraphics[height=.6cm]{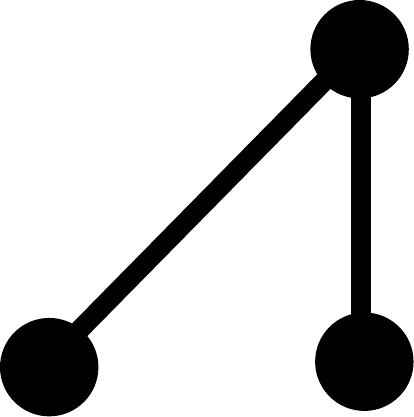}
\end{array},\text{ then the result of gluing $T'$ to $v$ is }\begin{array}{l}
\includegraphics[height=1.53cm]{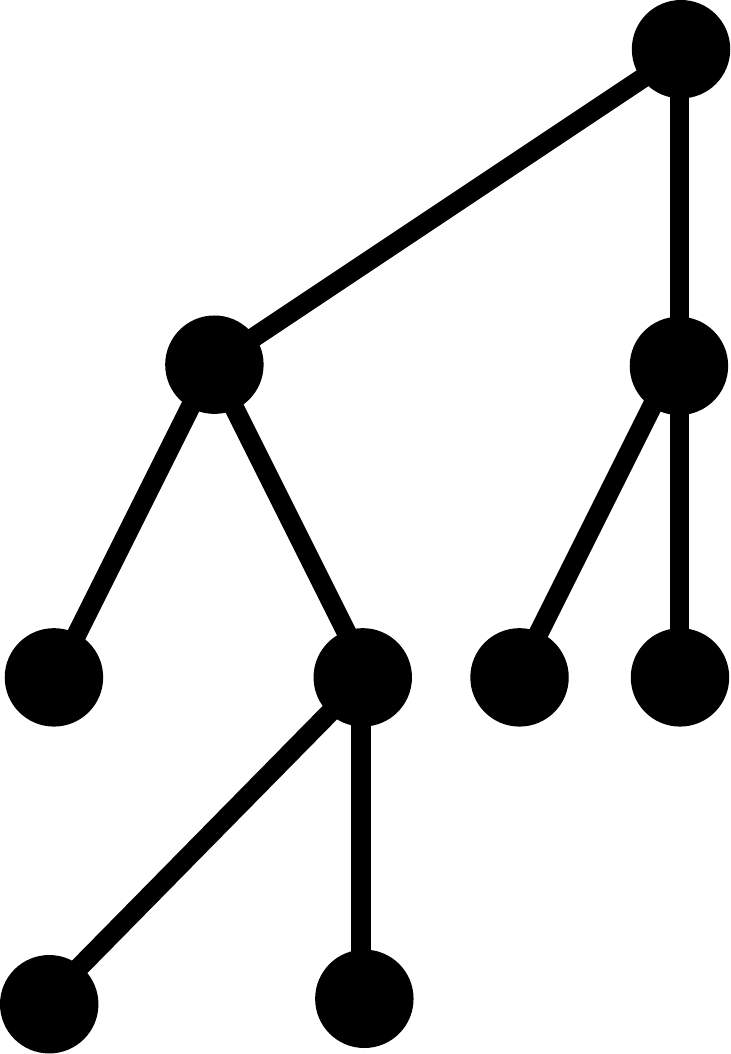}
\end{array}.\]

\subsection{Main Results}

Let $\eta_2=1$, and let $\eta_d=\frac{1}{2}$ for every $d\geq 3$. In Section~\ref{sec:0,1,...,d}, we will define numbers $\rho_d$, which arise as reciprocals of roots of certain polynomials. The purpose of the subsequent sections is to prove the following estimates (where $d\geq 2$ is a fixed integer):

\begin{enumerate*}
\item $N_{d\text{-ary}}^{\text{con}}(k)=d^{k-1}+k-1$;
\item $\eta_d\, k\log_2(k)(1+o(1))\leq N_{d\text{-ary}}^{\non}(k)\leq k^{\frac{1}{2}\log_2(k)(1+o(1))}$;
\item $d^{\frac{k-2}{d}}\leq N_{[d]}^{\con}(k)\leq (\rho_d+o(1))^k$;
\item  $\dfrac{\eta_d}{d}k\log_2(k)(1+o(1))\leq N_{[d]}^{\non}(k)\leq k^{\frac{1}{2}\log_2(k)(1+o(1))}$.
\end{enumerate*}

Some remarks are in order regarding these estimates. First of all, note that it is unusual to be able to prove an exact formula for the minimum size of a universal object, as we have done in (I). Next, a contiguous $k$-universal $d$-ary plane (respectively, $[d]$-) tree is certainly also a noncontiguous $k$-universal $d$-ary plane (respectively, $[d]$-) tree, so we trivially have $N_{d\text{-ary}}^{\text{con}}(k)\geq N_{d\text{-ary}}^{\text{non}}(k)$ and $N_{[d]}^{\con}(k)\geq N_{[d]}^{\non}(k)$.  Furthermore, one can change a $d$-ary plane tree into a $[d]$-tree by simply forgetting about empty subtrees and edge types. Doing so allows us to view a contiguous (respectively, noncontiguous) $k$-universal $d$-ary plane tree as a contiguous (respectively, noncontiguous) $k$-universal $[d]$-tree. Therefore, it follows from (I) that $N_{d\text{-ary}}^{\text{non}}(k)$, $N_{[d]}^{\text{con}}(k)$, and $N_{[d]}^{\text{non}}(k)$ are all at most $d^{k-1}+k-1$. However, the upper bounds in (II), (III), and (IV) greatly improve upon this observation. Indeed, the upper bounds in (II) and (IV) are subexponential in $k$, and the base of the exponential in the upper bound in (III) is much smaller than $d$. In Section~\ref{sec:0,1,...,d}, we will see that $\rho_d=1+\frac{4\log d}{d}(1+o(1))$ as $d\to\infty$. Compare this with the exponential lower bound in (III), in which the base of the exponential is $d^{1/d}=1+\frac{\log d}{d}(1+o(1))$. Finally, note that since $N_{[d]}^{\non}(k)\leq N_{[d]}^{\con}(k)$, we could deduce immediately from (III) that $N_{[d]}^{\non}(k)\leq (\rho_d+o(1))^k$. However, the subexponential upper bound in (IV) greatly improves upon this.  Similarly, the lower bound in (III) beats the lower bound in (IV).  We will show in Section~\ref{sec:0,1,...,d} that $N_{d\text{-ary}}^{\non}(k)$ and $N_{[d]}^{\non}(k)$ differ by at most a constant factor (for each fixed $d$); this fact explains why (II) and (IV) look similar.

Producing nontrivial lower bounds for the sizes of noncontiguous $k$-universal trees is fairly difficult; this is analogous to the permutation setting, where nontrivial lower bounds are scarce. For example, the best known lower bound for the length $n$ of a permutation that contains all length-$k$ patterns is given by $n\geq k^2/e^2$; this is a consequence of the simple observation that ${n\choose k}\geq k!$. The number of $d$-ary plane trees with $k$ vertices is $\frac{1}{(d-1)k+1}{dk\choose k}$, so a similar argument in our setting shows that ${N_{d\text{-ary}}^{\text{non}}(k)\choose k}\geq\frac{1}{(d-1)k+1}{dk\choose k}$. This translates to a lower bound of roughly $dk$ for $N_{d\text{-ary}}^{\text{non}}(k)$. Although many of the lower bounds in the (noncontiguous) permutation setting are trivial, there is a noteworthy exception: Albert, Engen, Pantone, and Vatter managed to obtain an explicit formula for the minimum length of a permutation that (noncontiguously) contains all length-$k$ layered permutations. Making use of a bijection between $231$-avoiding permutations and binary plane trees, we will invoke this result in order to prove the lower bound in (II).

\section{$d$-ary plane trees}\label{sec:d-ary}

\subsection{Contiguous containment}\label{subsec:d-ary-contiguous}

In the case of contiguous containment for $d$-ary plane trees, we obtain the exact size of the smallest $k$-supertree.  The upper bound comes from an explicit construction, and the lower bound comes from considering the family of paths.

\begin{theorem}\label{thm:d-ary-contiguous}
For all integers $d \geq 2$ and $k \geq 1$, we have $N_{d\ary}^{\con}(k)=d^{k-1}+k-1$.
\end{theorem}

\begin{proof}
We first show that $d^{k-1}+k-1$ is a lower bound for the size of a $k$-supertree. Let $\bf T$ be a contiguous $k$-universal $d$-ary plane tree.  Let $T_1,\ldots,T_{d^{k-1}}$ be the $d$-ary plane trees on $k$ vertices in which each nonleaf vertex has exactly one child (i.e., the $d$-ary plane trees that are paths on $k$ vertices).  For each $i\in\{1,\ldots,d^{k-1}\}$, there is a contiguous embedding $T_i^{\ast}$ of $T_i$ in $\bf T$.  Let $v_i^{\ast}$ denote the vertex in $\bf T$ that corresponds to the unique leaf of $T_i$ under this embedding. Starting at $v_i^*$ and tracing up $k-1$ edges, we immediately recover all of the edges of $T_i^*$, so the location of $v_i^{\ast}$ in $\bf T$ completely determines the isomorphism class of $T_i^{\ast}$. Because the trees $T_1^*,\ldots,T_{d^{k-1}}^*$ are pairwise nonisomorphic, the vertices $v_1^*,\ldots,v_{d^{k-1}}^*$ are pairwise distinct.  Thus, $\bf T$ contains at least $d^{k-1}$ vertices at depth at least $k-1$. Since $\bf T$ contains vertices at depth $k-1$, it must also contain at least one vertex at each depth $j$ for $0 \leq j \leq k-2$. This gives at least $k-1$ additional vertices, so $\bf T$ contains at least $d^{k-1}+k-1$ vertices, as desired.

We now construct a contiguous $k$-universal $d$-ary plane tree $\Delta_d(k)$ on exactly $d^{k-1}+k-1$ vertices. First, the tree $\Delta_d(1)$ consists of a single vertex. Now, consider $k\geq 2$. To construct $\Delta_d(k)$, first consider the $d$-ary plane tree that is a path on $k-1$ vertices in which every edge is of type $1$. Let $v$ denote the unique leaf of this path.  For each $2 \leq i \leq d$, attach a copy of the perfect tree $P_{k-2}^{(d)}$ in the $i^\text{th}$ subtree of $v$.  (Recall the definition of perfect trees from the introduction.)  Note that each of these $d-1$ added trees contains exactly $1+d+\cdots+d^{k-2}$ vertices, which means that together they have $d^{k-1}-1$ vertices in total.  Next, consider the leftmost leaf in the copy of $P_{k-2}^{(d)}$ that is sitting in the second subtree of $v$.  Add one more vertex in the first subtree of this leaf.  The resulting tree $\Delta_d(k)$ has the desired number of vertices.  See Figure \ref{Fig1} for an example.

\begin{figure}[h]
\begin{center} 
\includegraphics[width=0.65\linewidth]{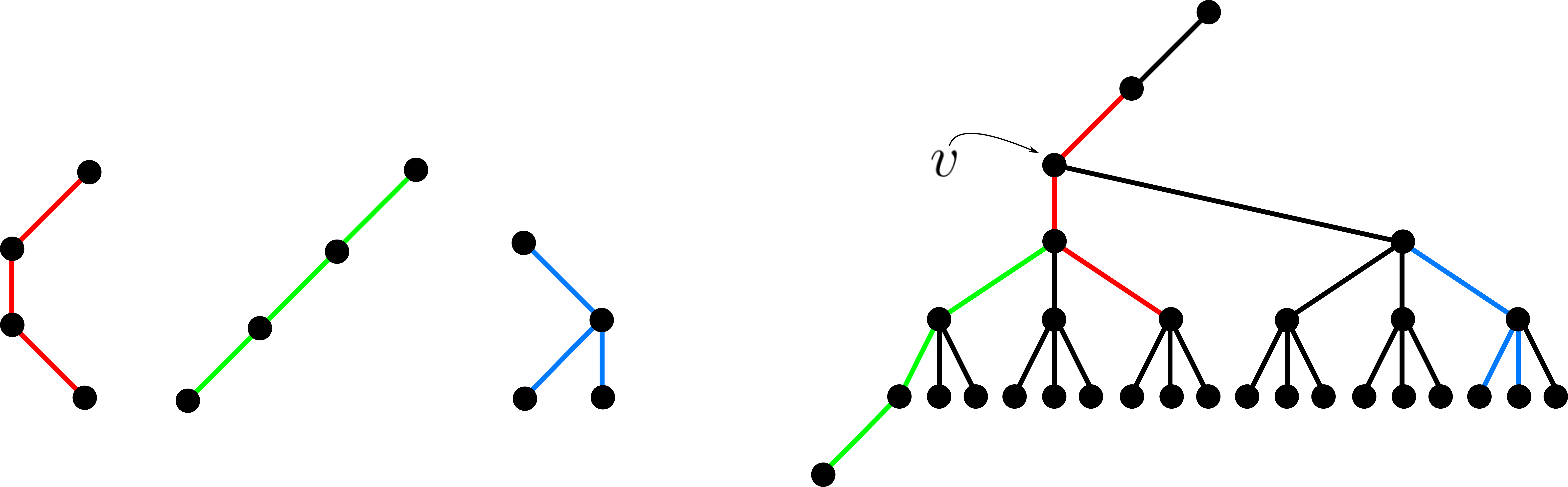}
\end{center}
\caption{The tree $\Delta_3(4)$ is depicted on the right. This tree contiguously contains all $3$-ary plane trees on $4$ vertices, three of which are shown on left.}\label{Fig1}
\end{figure}

Finally, we show that $\Delta_d(k)$ is in fact a $k$-supertree.  Fix any $d$-ary plane tree $T$ with $k$ vertices.  If $T$ is not a path, then it has height at most $k-2$, so it fits into one of the copies of $P_{k-2}^{(d)}$.  If $T$ is the path whose edges are all of type $1$, then we can embed $T$ in $\Delta_d(k)$ by mapping the root of $T$ to the root of the second subtree (i.e., the leftmost nonempty subtree) of $v$. Now, suppose $T$ is a path in which at least one edge is not of type $1$. Let $m$ be the smallest element of $\{1,\ldots,k-1\}$ such that the $m^\text{th}$ edge from the top of $T$ is not of type $1$. We can embed $T$ into $\Delta_d(k)$ by mapping the unique vertex in $T$ of depth $m-1$ to $v$.  This exhausts all cases and shows that $\Delta_d(k)$ is $k$-universal.
\end{proof}

\subsection{Noncontiguous containment}\label{subsec:d-ary-noncontiguous}

\subsubsection{Lower bounds}\label{subsubsec:noncontiguous-lower}

Recall that $\eta_2=1$ and $\eta_d=\frac{1}{2}$ for all $d\geq 3$. In this subsection we will prove the following theorem.

\begin{theorem}\label{Thm3}
For all integers $d \geq 2$ and $k\geq 1$, we have \[N_{d\ary}^{\non}(k)\geq \eta_d\left((k+1)\left\lceil\log_2(k+1)\right\rceil-2^{\left\lceil\log_2(k+1)\right\rceil}+1\right).\]
\end{theorem}

The first step is to show that it suffices to consider the specific case in which $d=2$. 

\begin{proposition}\label{Prop1}
For all integers $d\geq 3$ and $k\geq 1$, we have \[N_{d\ary}^{\non}(k)>\frac{1}{2}N_{2\ary}^{\non}(k).\]
\end{proposition}

\begin{proof}
Fix $d\geq 3$, and consider the $d$-ary plane tree path on $m-1$ vertices in which every edge has type $1$. Let $\vec{t}=(t_1,\ldots, t_m)$ be an $m$-tuple of integers satisfying $1\leq t_1<\cdots<t_m\leq d$.  For $2 \leq i \leq m$, attach a single child via an edge of type $t_i$ to the $(i-1)^{\text{th}}$ vertex of the path, counting from the bottom. Then attach a single child via an edge of type $t_1$ to the bottom vertex of this original path.  Call the resulting $d$-ary plane tree $J_{\vec{t}}$.

Let $\mathcal T = \mathcal T_0$ be a $d$-ary plane tree. We will transform $\mathcal T$ into a tree $\widetilde{\mathcal T}$ in which every vertex has at most two children. Suppose $\mathcal T$ has a vertex $v_1$ with at least $3$ children, say exactly $m_1$ children in subtrees of type $\vec{t}_1=(t_{1,1},\cdots,t_{1,m_1})$.  Replace $v_1$ with a copy of $J_{\vec{t}_1}$ in the following manner. Detach the subtrees of $v_1$, and glue $J_{\vec{t}_1}$ to $v_1$. Then glue the (detached) $i^\text{th}$ nonempty subtree (counted from the left) of $v_1$ to the $i^\text{th}$ leaf (again counted from the left) of the copy of $J_{\vec{t}_1}$. Call this tree $\mathcal T_1$. Choose another vertex $v_2$ that has $m_2\geq 3$ children in subtrees of type $\vec{t}_2$, and replace it in the same fashion with a copy of $J_{\vec{t}_2}$ to obtain $\mathcal T_2$. Continue this process until reaching a tree $\widetilde{\mathcal T} = \mathcal T_r$ in which each vertex has at most $2$ children.

Note that if $1\le i\le r$, then $\mathcal T_i$ noncontiguously contains $\mathcal T_{i-1}$ because we can legally contract the edges of the original path in the added $J_{\vec{t}_i}$ from top to bottom.  Iterating this procedure shows that there is a sequence of legal contractions that begins with $\widetilde{\mathcal{T}} = \mathcal T_r$ and ends with $\mathcal T_0 = \mathcal T$.

We can naturally associate $\widetilde{\mathcal T}$ with a \emph{binary} plane tree $\widetilde{\mathcal T}'$. If there is an only child of type $1$ in $\widetilde{\mathcal T}$, it becomes an only child of type $1$ in $\widetilde{\mathcal T}'$. If there is an only child of type other than $1$ in $\widetilde{\mathcal T}$, it becomes an only child of type $2$ in $\widetilde{\mathcal T}'$. See Figure \ref{Fig4} for an example when $d=4$. As proven above, our construction guarantees that $\widetilde{\mathcal T}$ noncontiguously contains $\mathcal T$, so it also noncontiguously contains every $d$-ary plane tree that $\mathcal T$ noncontiguously contains.

\begin{figure}[h]
\begin{center}
\includegraphics[width=.76\linewidth]{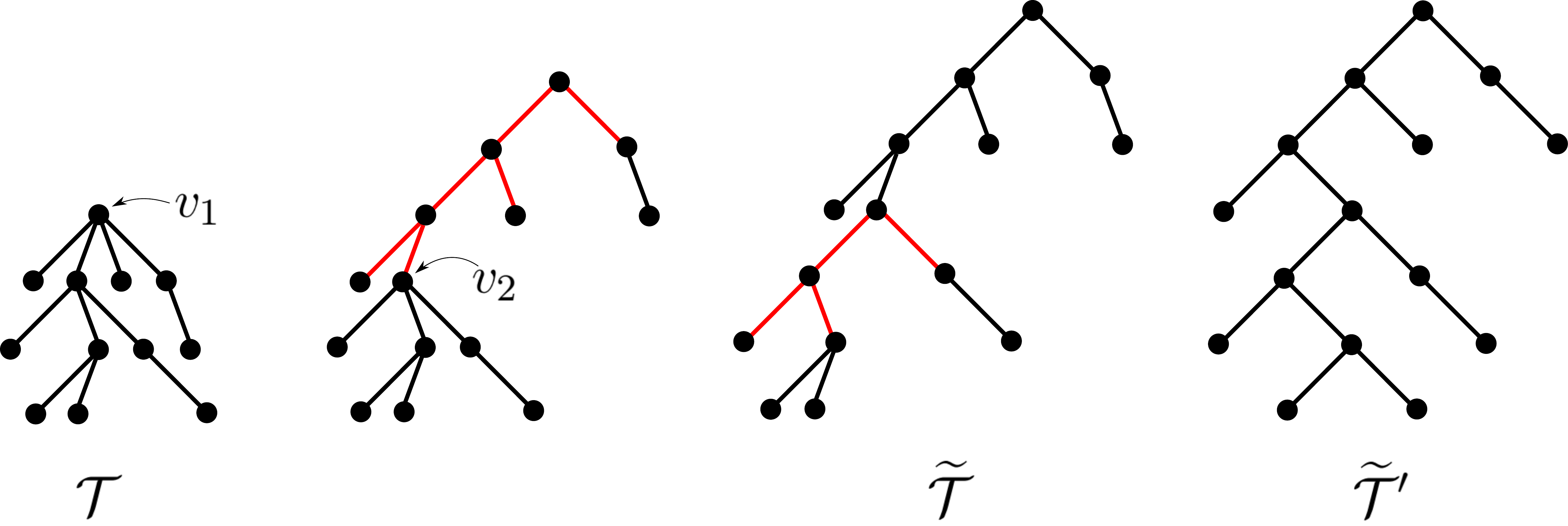}
\caption{Transforming $\mathcal T$ into $\widetilde{\mathcal T}'$. We have used the color red to indicate the edges of the inserted copy of $J_{m_i}$ added in the $i^\text{th}$ step.}
\label{Fig4}
\end{center}  
\end{figure}

Let $\bf T$ be a noncontiguous $k$-universal $d$-ary plane tree with $\alpha=N_{d\ary}^{\non}(k)$ vertices. The tree $\widetilde{\bf T}$ obtained via the above construction is also a noncontiguous $k$-universal $d$-ary plane tree, so $\widetilde{\bf T}'$ is a noncontiguous $k$-universal binary plane tree. The trees $\widetilde{\bf T}$ and $\widetilde{\bf T}'$ have the same number of vertices, say $\beta$. We know that $\beta\geq N_{2\ary}^{\non}(k)$. We will show that $\beta<2\alpha$, establishing the desired result. 

Let $f_r$ denote the number of vertices in $\bf T$ with exactly $r$ children. We obtained $\widetilde{\bf T}$ from $\bf T$ by substituting a copy of $J_m$ for each vertex $v$ of $\bf T$ with $m\geq 3$ children. Note that each such substitution increased the number of vertices in the tree by $m-2$. Thus, $\beta=\alpha+\sum_{m=3}^d(m-2)f_m$. We know that $\sum_{m=0}^d f_m=\alpha$. Furthermore, counting the $\alpha-1$ edges in $\bf T$ according to the number of children of their parent vertices gives $\alpha-1=\sum_{m=0}^dmf_m$. Consequently,
\begin{align*}
    \beta &=\alpha+\sum_{m=3}^d(m-2)f_m =\alpha+2f_0+f_1+\sum_{m=0}^d(m-2)f_m\\
     &=\alpha+2f_0+f_1+(\alpha-1)-2\alpha =2f_0+f_1-1<2\sum_{m=0}^df_m=2\alpha. \qedhere
\end{align*}
\end{proof}

For the proof of Theorem \ref{Thm3}, it now remains only to show that 
\begin{equation}\label{Eq3}N_{2\text{-ary}}^{\text{non}}(k)\geq (k+1)\left\lceil\log_2(k+1)\right\rceil-2^{\left\lceil\log_2(k+1)\right\rceil}+1.
\end{equation}
Let us first establish some terminology and notation concerning labeled trees and tree traversals. Let $\PT_n^{(2)}$ denote the set of binary plane trees with $n$ vertices. A \emph{decreasing binary plane tree} is a binary plane tree whose vertices are labeled with distinct positive integers so that the label of each nonroot vertex is smaller than the label of its parent. Let $\DPT_n^{(2)}$ be the set of decreasing binary plane trees with $n$ vertices in which the labels form the set $[n]$. We can read the labels of a decreasing binary plane tree in \emph{in-order} by first reading the labels of the left subtree of the root in in-order, then reading the label of the root, and finally reading the labels of the right subtree of the root in in-order. Let $I(\Upsilon)$ denote the in-order reading of the decreasing binary plane tree $\Upsilon$. The map $I:\DPT_n^{(2)}\to S_n$ is a bijection \cite[Chapter 8]{Bona}. Alternatively, we can read the labels of a decreasing binary plane tree in \emph{postorder} by first reading the labels of the left subtree of the root in postorder, then reading the labels of the right subtree of the root in postorder, and finally reading the label of the root. 

For each unlabeled tree $T\in\PT_n^{(2)}$, there is a unique way to label the vertices of $T$ so that the resulting labeled tree $\omega(T)\in\DPT_n^{(2)}$ has postorder reading $123\cdots n$ (the increasing permutation). This gives us a map $\omega:\PT_n^{(2)}\to\DPT_n^{(2)}$. Let $\psi(T)=I(\omega(T))$. It is not difficult to check that the permutation $\psi(T)$ avoids the pattern $231$. In fact, we have the following useful proposition.

\begin{proposition}\label{Prop2}
The map $\psi$ is a bijection from the set of binary plane trees with $n$ vertices to the set of $231$-avoiding permutations in $S_n$. If $\mathcal T$ is a binary plane tree that noncontiguously contains the binary plane tree $T$, then the permutation $\psi(\mathcal T)$ contains the pattern $\psi(T)$. 
\end{proposition}

\begin{proof}
The map $\psi$ is injective because $\omega$ and $I$ are injective. The first statement of the proposition now follows from the fact that the number of binary plane trees with $n$ vertices and the number of $231$-avoiding permutations in $S_n$ are both equal to the $n^\text{th}$ Catalan number.\footnote{The first statement of this proposition is not new; it is essentially equivalent to the fact that a permutation is $1$-stack-sortable if and only if it avoids $231$ (see one of the references \cite{Bona, DefantPolyurethane, DefantPostorder} for more details).}  

To prove the second statement, we need to understand the effect of legal edge contractions on the corresponding permutations. Let $\mathcal T\in\PT_n^{(2)}$ be a binary plane tree, and let $e$ be an edge of $\mathcal T$ that can be legally contracted. Let $a$ and $b$ be, respectively, the labels of the upper and lower endpoints of $e$ in $\omega(\mathcal T)$. Let $\mathcal T/e\in\PT_{n-1}^{(2)}$ denote the tree that is obtained by contracting the edge $e$ in $\mathcal T$. One can check that if $e$ is a type-$1$ edge, then $\psi(\mathcal T/e)$ is the permutation obtained by deleting the entry $b$ from $\psi(\mathcal T)$ and then normalizing to obtain a permutation in $S_{n-1}$. Similarly, if $e$ is a type-$2$ edge, then $\psi(\mathcal T/e)$ is the permutation obtained by deleting the entry $a$ from $\psi(\mathcal T)$ and then normalizing. In either case, $\psi(\mathcal T)$ contains $\psi(\mathcal T/e)$ as a pattern. If $\mathcal T$ noncontiguously contains a binary plane tree $T$ (meaning $T$ is obtained from $\mathcal T$ via a sequence of legal edge contractions), then $\psi(\mathcal T)$ contains $\psi(T)$ as a pattern. 
\end{proof}

Let us illustrate the proof of the second statement of Proposition \ref{Prop2} with an example. If \[T\text{ is }\begin{array}{l}
\includegraphics[height=2.6cm]{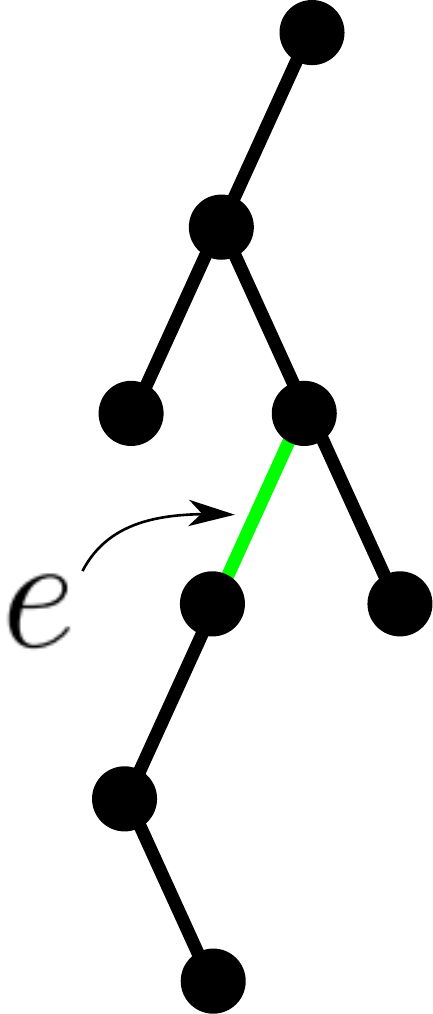}
\end{array},\text{ then }\omega(T)\text{ is }\begin{array}{l}
\includegraphics[height=2.6cm]{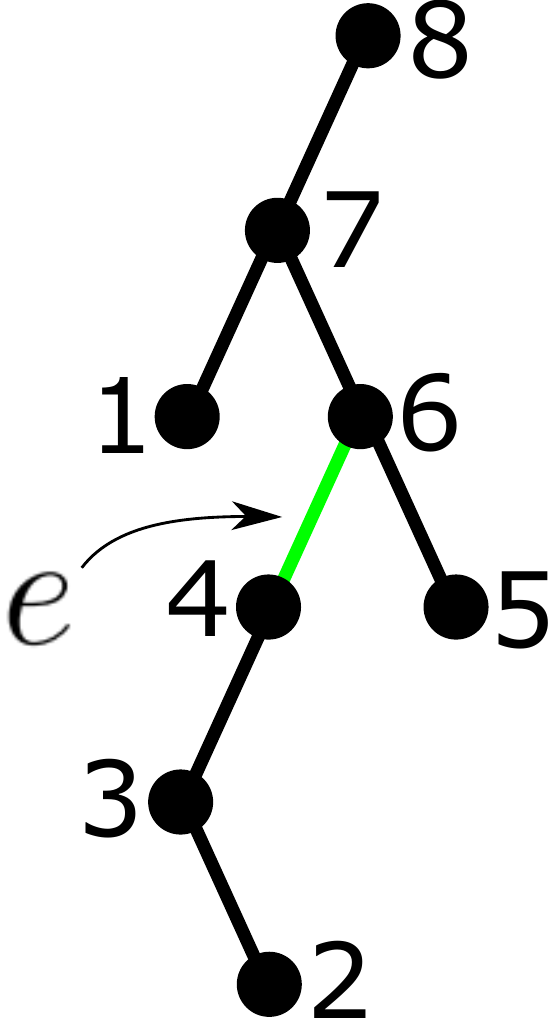}
\end{array},\text{ and }\psi(T)=I(\omega(T))=17324658.\] Contracting the edge labeled $e$, we find that \[T/e\text{ is }\begin{array}{l}
\includegraphics[height=2.2cm]{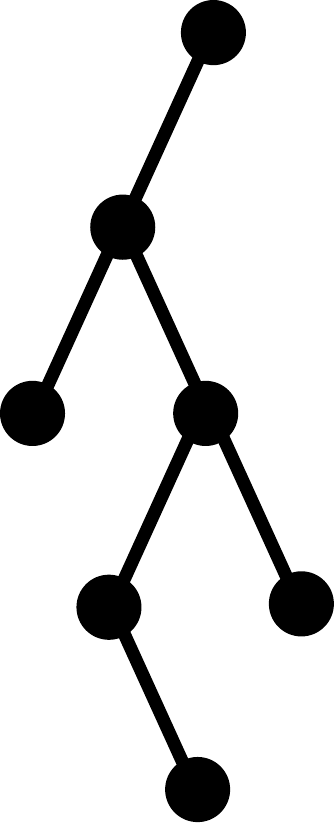}
\end{array},\text{ }\omega(T/e)\text{ is }\begin{array}{l}
\includegraphics[height=2.2cm]{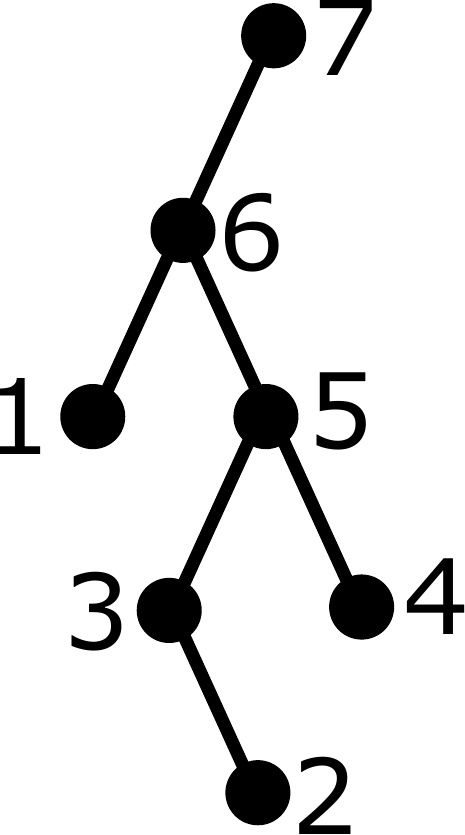}
\end{array},\text{ and }\psi(T/e)=I(\omega(T/e))=1632547.\] Note that since $e$ is a left edge, the permutation $\psi(T/e)=1632547$ is obtained by deleting the entry $b=4$ from the permutation $\psi(T)=17324658$ and then normalizing. 

We can finally deduce inequality \eqref{Eq3}.  Suppose $\mathcal T$ is a noncontiguous $k$-universal binary plane tree with $N_{2\text{-ary}}^{\text{non}}(k)$ vertices. Proposition \ref{Prop2} tells us that $\psi(\mathcal T)$ contains every $231$-avoiding permutation in $S_k$. A permutation is called \emph{layered} if it avoids both $231$ and $312$. Thus, $\psi(\mathcal T)$ is a permutation of length $N_{2\text{-ary}}^{\text{non}}(k)$ that contains all layered permutations in $S_k$. The authors of \cite{Albert} proved that the minimum size of a permutation that contains all layered permutations in $S_k$ is $(k+1)\left\lceil\log_2(k+1)\right\rceil-2^{\left\lceil\log_2(k+1)\right\rceil}+1$. This establishes \eqref{Eq3} and hence completes the proof of Theorem \ref{Thm3}.

\subsubsection{Upper bounds}\label{subsubsec:noncontiguous-upper}

For every $d\geq 2$ and $k\geq 1$, we now construct a noncontiguous $k$-universal $d$-ary plane tree $\xi_d(k)$.  The construction is natural but fairly intricate.  We begin by defining a few specific $d$-ary plane trees that will form the building blocks in our construction of $\xi_d(k)$.  The \textit{$d$-crescent} is the path on $d+1$ vertices in which the vertex at depth $i$ is connected to its parent by an edge of type $i$. Now, take three copies of the $d$-crescent. Remove the lowest vertex from the first of these $d$-crescents, and glue the remaining tree to the vertex of depth $1$ in the second crescent. Next, remove the root vertex from the third $d$-crescent, and glue the remaining tree to the root of the second $d$-crescent. We call the resulting tree the \textit{$d$-vertebra} and denote it by $V_d$.  Note that $V_d$ has exactly $3$ leaves, which we call the \textit{left}, \textit{center}, and \textit{right} leaves (in the obvious fashion).

\begin{figure}[h]
\begin{center} 
\includegraphics[height=4.5cm]{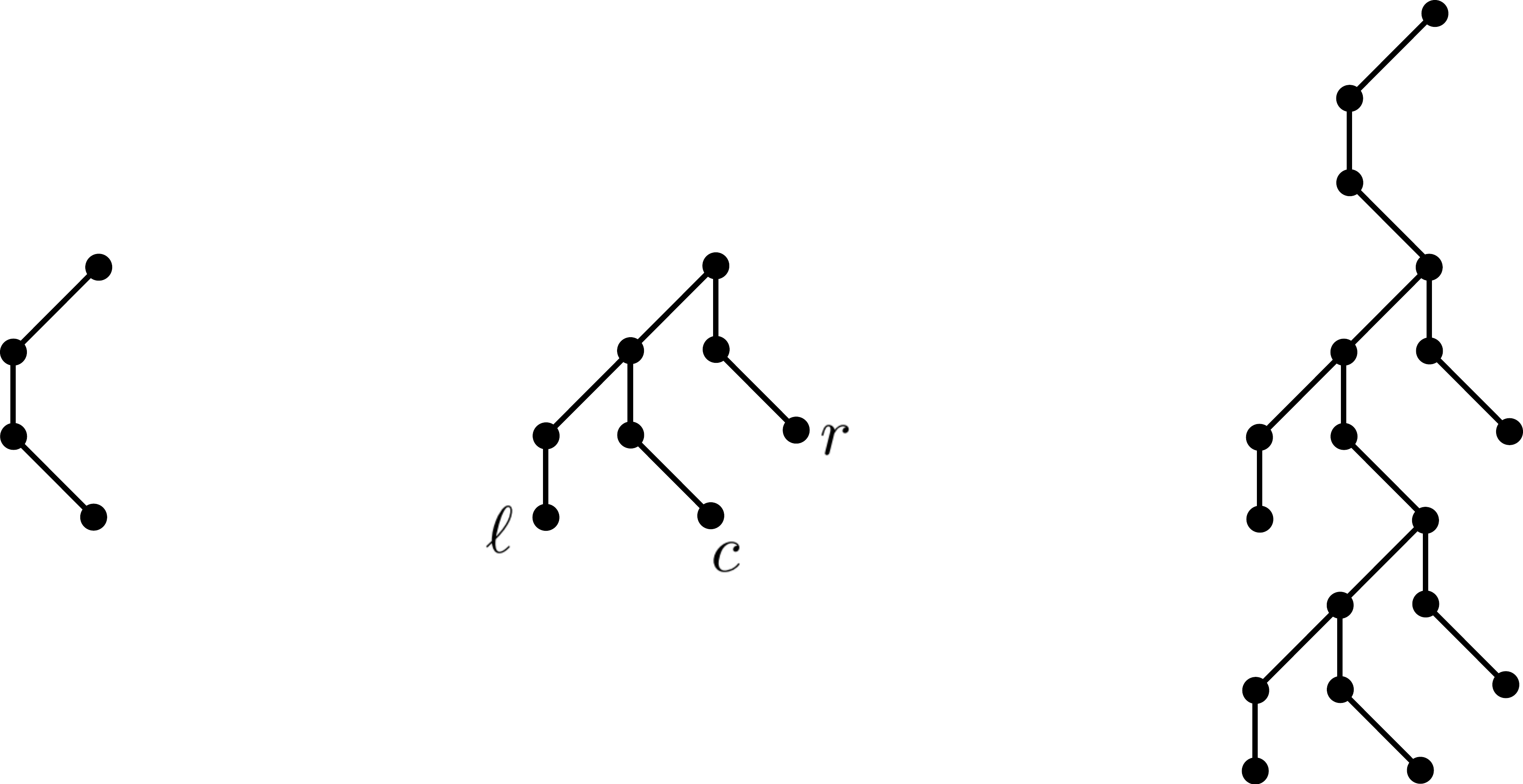}
\end{center}
\caption{From left to right: the $3$-crescent; the $3$-vertebra $V_3$, with the left, center, and right leaves labeled $\ell$, $c$, and $r$ (respectively); and the $2^{\text{nd}}$ $3$-spine.}\label{Fig2}
\end{figure}

For $m \geq 1$, we obtain the \textit{$m^{\textit{th}}$ $d$-spine} by consecutively gluing  $m$ copies of the $d$-vertebra $V_d$ under a single copy of the $d$-crescent: the first $V_d$ is glued to the single leaf of the $d$-crescent, and each subsequent $V_d$ is glued to the center leaf of the previous $V_d$.  We speak of the first, second, etc. $d$-vertebra beginning with the highest one.

At last, we recursively define the families $\xi_d(k)$.  We first describe the following base cases:
\begin{itemize}
    \item Let $\xi_d(1)$ consist of a single vertex.
    \item Let $\xi_d(2)$ be the $d$-crescent.
    \item Obtain $\xi_d(3)$ from a $d$-crescent by giving the leaf $d$ children (one in each position).
\end{itemize}
The construction for larger $k$ is recursive and differs for $d=2$ and $d>2$.  (The $d=2$ construction is a slight improvement on the $d>2$ construction.)  If $d=2$, then for $k \geq 4$, we obtain $\xi_2(k)$ from the $\left(\left\lfloor \frac{k}{2}\right\rfloor-1\right)^{\text{th}}$ $2$-spine as follows:
\begin{enumerate}
    \item For each $1 \leq i \leq \left\lfloor \frac{k}{2}\right\rfloor-2$, glue a copy of $\xi_2(i)$ to each of the left and right leaves of the $i^{\text{th}}$ $2$-vertebra.
    \item Glue a copy of $\xi_2(\left\lfloor \frac{k}{2}\right\rfloor-1)$ to the right leaf of the $\left(\left\lfloor \frac{k}{2}\right\rfloor-1\right)^{\text{th}}$ (i.e., lowest) $2$-vertebra.
    \item Glue a copy of $\xi_2(\left\lceil\frac{k}{2}\right\rceil-1)$ to the left leaf of the $\left(\left\lfloor \frac{k}{2}\right\rfloor-1\right)^{\text{th}}$ $2$-vertebra.
    \item Glue a copy of $\xi_2(\left\lceil\frac{k}{2}\right\rceil)$ to the center leaf of the $\left(\left\lfloor \frac{k}{2}\right\rfloor-1\right)^{\text{th}}$ $2$-vertebra.
\end{enumerate}

If $d>2$, then for $k \geq 4$, we obtain $\xi_d(k)$ from the $\left\lfloor \frac{k}{2}\right\rfloor^{\text{th}}$ $d$-spine as follows:
\begin{enumerate}
    \item For each $1 \leq i \leq \left\lfloor \frac{k}{2}\right\rfloor-2$, glue a copy of $\xi_d(i)$ to each of the left and right leaves of the $i^{\text{th}}$ $d$-vertebra.
    \item Glue a copy of $\xi_d(\left\lfloor \frac{k}{2}\right\rfloor-1)$ to the right leaf of the $\left(\left\lfloor \frac{k}{2}\right\rfloor-1\right)^{\text{th}}$ (i.e., second-lowest) $d$-vertebra.
    \item Glue a copy of $\xi_d(\left\lceil\frac{k}{2}\right\rceil-1)$ to the left leaf of the $\left(\left\lfloor \frac{k}{2}\right\rfloor-1\right)^{\text{th}}$ $d$-vertebra.
    \item Glue a copy of $\xi_d(\left\lceil\frac{k}{2}\right\rceil)$ to the center leaf of the $\left\lfloor \frac{k}{2}\right\rfloor^{\text{th}}$ (i.e., lowest) $d$-vertebra.
    \item Glue a copy of $\xi_d\left( \left\lfloor \frac{k+1}{4} \right\rfloor \right)$ to each of the left and right leaves of the $\left\lfloor \frac{k}{2}\right\rfloor^{\text{th}}$ $d$-vertebra.
\end{enumerate}
For $k\geq 4$, the \textit{tail} of $\xi_d(k)$ is the copy of $\xi_d(\left\lceil \frac{k}{2} \right\rceil)$ that is glued to the center leaf of the bottom of the spine in step (4) (in both the $d=2$ and $d>2$ constructions).  Figure~\ref{Fig3} shows $\xi_2(k)$ for some small values of $k$.

\begin{figure}[h]
\begin{center} 
\includegraphics[height=6.5cm]{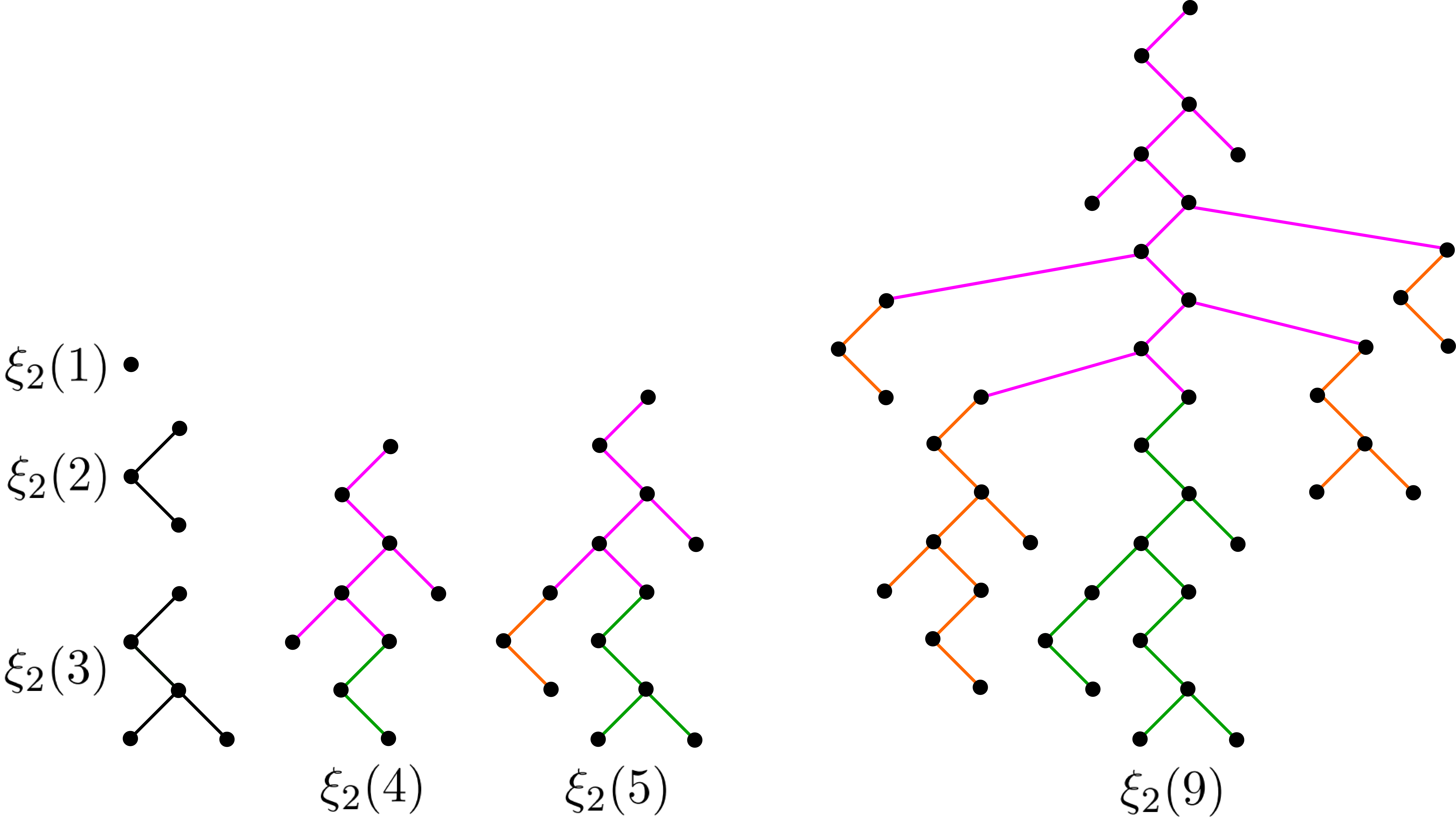}
\end{center}
\caption{The trees $\xi_2(k)$ for $1\leq k\leq 5$, along with $\xi_2(9)$. In $\xi_2(4)$, $\xi_2(5)$, and $\xi_2(9)$, the pink edges represent the spine. The orange edges represent the copies of the previously-constructed trees that are glued to the spine, and the green edges represent the tail.}\label{Fig3}
\end{figure}

We now show that $\xi_d(k)$ noncontiguously contains every $d$-ary plane tree with $k$ vertices.  The big-picture idea is that we can ``siphon off'' small subtrees of the tree that we are trying to contain until what remains fits into the tail.  Many of the arguments are the same for $d=2$ and $d>2$, so we present the proofs together.  The reader may find it helpful to bear in mind the example of $\xi_2(9)$ (as shown in Figure~\ref{Fig3}).

\begin{theorem}\label{thm:colorful-construction}
For all integers $d \geq 2$ and $k \geq 1$, the tree $\xi_d(k)$ noncontiguously contains every $d$-ary plane tree with $k$ vertices. 
\end{theorem}
\begin{proof}
Fix $d$.  We proceed by strong induction on $k$.  The statement is obviously true for $k \leq 3$.  Now, consider $k \geq 4$.  Let $T$ be a $d$-ary plane tree on $k$ vertices.  We will show that $\xi_d(k)$ noncontiguously contains $T$ by showing that $\xi_d(k)$ noncontiguously contains a larger tree $T'$, which in turn noncontiguously contains $T$.


We construct $T'$ from $T$ by defining a finite sequence of pairs $(T_i,v_i)$, where $T_i$ is a tree and $v_i$ is a vertex of $T_i$; we then let $T'$ be the last $T_i$. We will see that we can naturally view the vertices $v_0,\ldots,v_i$ as vertices in the tree $T_{i+1}$. In particular, we can view all of the vertices $v_i$ as vertices in the last tree $T'$. First, let $T_0=T$, and let $v_0$ be the root of $T_0$.  If at any time the subtree in $T_i$ below $v_i$ (including $v_i$ itself) contains at most $\left\lceil \frac{k}{2} \right\rceil$ vertices, then the sequence terminates.  As long as this situation is not achieved, we obtain $(T_{i+1},v_{i+1})$ from $(T_i,v_i)$ as follows.  If $v_i$ has only a single child, then we let $v_{i+1}$ denote this child and let $T_{i+1}=T_i$.  If $v_i$ has exactly $2$ children, then we let $v_{i+1}$ denote the child with the larger subtree (breaking ties with preference for the right child) and let $T_{i+1}=T_i$.

Otherwise, $v_i$ has at least $3$ children.  (This possibility of course pertains only to $d>2$.)  We consider the leftmost and rightmost nonempty subtrees of $v_i$ and obtain $T_{i+1}$ by performing the following operations.  If the leftmost nonempty subtree contains fewer vertices than the rightmost nonempty subtree or these two subtrees contain the same number of vertices, then detach all of the subtrees of $v_i$ except the leftmost nonempty one, add a new child $v_{i+1}$ in the $d^\text{th}$ subtree of $v_i$ via a red edge of type $d$, and reattach the detached subtrees as new subtrees of $v_{i+1}$ (so that the reattached edges have the same types that they originally had).  If the rightmost nonempty subtree contains fewer vertices than the leftmost nonempty subtree, then detach all of the subtrees of $v_i$ except the rightmost nonempty one, add a new child $v_{i+1}$ in the $1^\text{st}$ subtree of $v_i$ via a red edge of type $1$, and reattach the detached subtrees as new subtrees of $v_{i+1}$. 

\begin{figure}[h]
\begin{center} 
\includegraphics[width=0.7\linewidth]{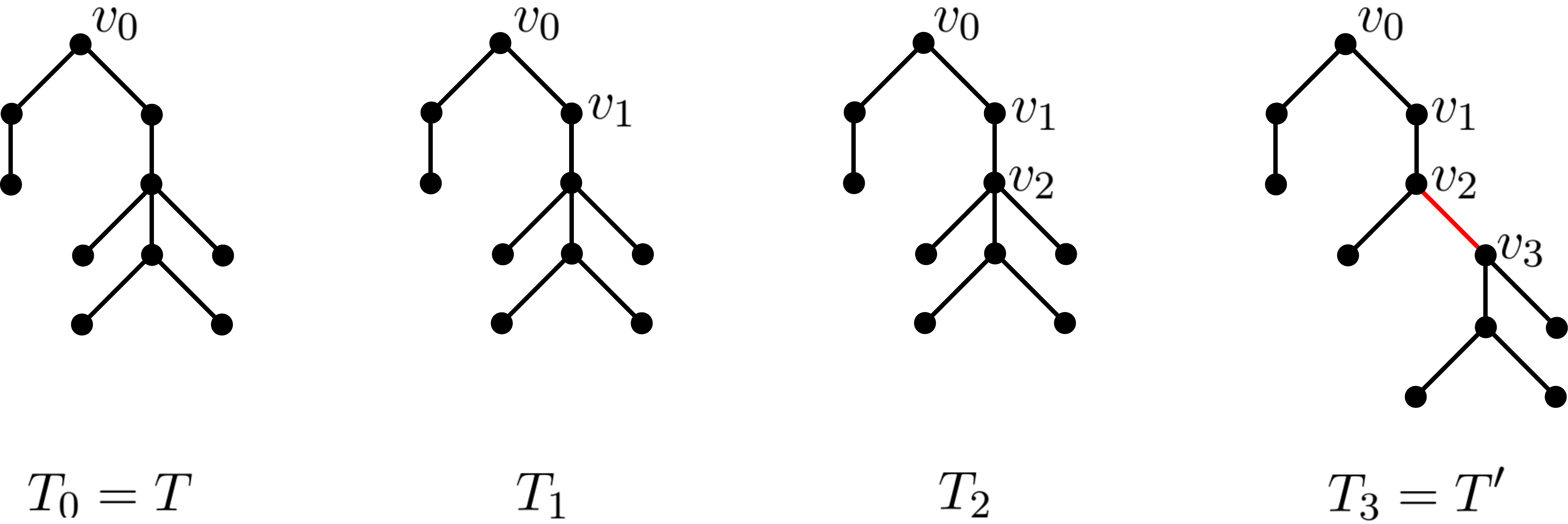}
\end{center}
\caption{An illustration of the sequence transforming $T$ into $T'$, where $d=3$ and $k=10$. We also have $s_0=2$, $s_1=0$, and $s_2=1$.}\label{Fig5}
\end{figure}

This sequence terminates in some $(T_m,v_m)$ with $1\leq m\leq \left\lfloor \frac{k}{2} \right\rfloor$ since each $v_{i+1}$ has strictly fewer vertices below it than $v_i$.  Note that each $T_i$ is either the same as $T_{i+1}$ or else can be obtained from $T_{i+1}$ by (legally) contracting the added red edge between $v_i$ and $v_{i+1}$.  In particular, $T'$ noncontiguously contains $T$. (When $d=2$, $T'$ equals $T$ because we did not add any red edges.) Each vertex $v_i$, for $0 \leq i \leq m-1$, has at most $2$ children in $T'$.  When $v_i$ has exactly $2$ children in $T'$, we think of the subtree containing $v_{i+1}$ as continuing down the main ``trunk'' of $T'$ and the other (smaller) subtree, which we call $\tau_i$, as ``branching off.''  In this case, we let $s_i=|\tau_i|$. If $v_i$ has only $1$ child, then we let $s_i=0$.  Note that the difference between the number of non-red edges below $v_{i}$ and the number of non-red edges below $v_{i+1}$ is given by
$$\begin{cases}
1, &\text{if }v_i\text{ has only a single child in }T'\\
s_i, &\text{if }v_i\text{ has two children in }T'\text{ and the edge between }v_i\text{ and }v_{i+1}\text{ is red}\\
s_i+1, &\text{if }v_i\text{ has two children in }T'\text{ and the edge between }v_i\text{ and }v_{i+1}\text{ is not red}.
\end{cases}$$
We can write this number of non-red edges more concisely as
\begin{equation} \label{eq:condition}
\begin{cases}
\max \{1,s_i\}+1, &\text{if }v_i\text{ has two children and the edge between }v_i\text{ and }v_{i+1}\text{ is not red}\\
\max \{1,s_i\}, &\text{otherwise}.
\end{cases}
\end{equation}
This characterization will be useful later.  Now, we condition on $m$: if $m=1$, then it is in fact easier to embed $T$ in $\xi_d(k)$ directly; if $m>1$, then we describe the embedding of $T'$ in $\xi_d(k)$.

First, suppose $m=1$, i.e., the algorithm terminates after a single step.  If $k$ is even, then the root of $T$ must have two children, with $\frac{k}{2}$ and $\frac{k}{2}-1$ vertices, respectively.  We identify the root of $T$ with the top vertex of the $\left(\frac{k}{2}-1\right)^{\text{th}}$ vertebra.  If the subtree with $\frac{k}{2}-1$ vertices is the right (respectively, left) child of the root, then we can embed this subtree in the copy of $\xi_d\left(\frac{k}{2}-1 \right)$ that is attached to the right (respectively, left) leaf of the $\left(\frac{k}{2}-1\right)^{\text{th}}$ vertebra; and we embed the subtree with $\frac{k}{2}$ vertices in the copy of $\xi_d\left( \frac{k}{2} \right)$ in the tail.  (We have used the inductive hypothesis that $\xi_d(\kappa)$ is actually a $\kappa$-supertree for all $\kappa<k$.)  If $k$ is odd, then there are two possibilities for the subtrees of the root of $T$.
\begin{enumerate}[(i)]
    \item There are two subtrees, each with $\frac{k-1}{2}$ vertices.  We identify the root of $T$ with the top vertex of the $\left(\frac{k-3}{2}\right)^{\text{th}}$ vertebra.  We can embed the left subtree in the copy of $\xi_2\left(\frac{k-1}{2} \right)$ that is attached to the left leaf of the $\left(\frac{k-3}{2}\right)^{\text{th}}$ vertebra, and we can embed the right subtree in the tail.
    \item There are two subtrees, with $\frac{k-3}{2}$ and $\frac{k+1}{2}$ vertices, respectively.  If the subtree with $\frac{k-3}{2}$ vertices is the right (respectively, left) child of the root, then we can embed this subtree in the right (respectively, left) tree that is glued to the $\left(\frac{k-3}{2}\right)^{\text{th}}$ vertebra; and we embed the subtree with $\frac{k+1}{2}$ vertices in the tail.
\end{enumerate}

We now turn to the case $m>1$.  We will describe how to noncontiguously embed $T'$ into $\xi_d(k)$.  We first define functions $f_2:\{0,\ldots,m\}\to \{0,\ldots,\left\lfloor \frac{k}{2}\right\rfloor\}$ and $f_{>2}:\{0,\ldots,m\}\to \{0,\ldots,\left\lfloor \frac{k}{2}\right\rfloor+1\}$ that, roughly speaking, tell us how far down $\xi_d(k)$ to embed each $v_i$.  Unsurprisingly, $f_2$ will be for the $d=2$ case, and $f_{>2}$ will be for the $d>2$ case.  In what follows, we will write $f_*$ in statements that apply to both $f_2$ and $f_{>2}$.  Let $f_2(0)=f_{>2}(0)=s_0$.  For $1 \leq i \leq m-1$, let $$f_2(i)=\max\{f_2(i-1)+1,s_i\}\quad\text{and}\quad f_{>2}(i)=\max\{f_{>2}(i-1)+1,s_i\}.$$  Finally, let $f_2(m)=\left\lfloor \frac{k}{2} \right\rfloor$ and $f_{>2}(m)=\left\lfloor \frac{k}{2} \right\rfloor+1$.  We will see that $f_*$ is strictly increasing and in fact has the claimed codomain; before establishing these facts, we show that they will let us embed $T'$ in $\xi_d(k)$.

For each $0 \leq i \leq m$, we identify $v_i$ with a vertex of $\xi_d(k)$ as follows.  For the $d=2$ case, we identify $v_i$ with: the root of $\xi_2(k)$ if $f_2(i)=0$; the topmost vertex in the $f_2(i)^{\text{th}}$ vertebra of $\xi_2(k)$ if $1 \leq f_2(i) \leq \left\lfloor \frac{k}{2}\right\rfloor-1$; and the topmost vertex of the tail if $f_2(i)=\left\lfloor \frac{k}{2}\right\rfloor$.  For the $d>2$ case, we identify $v_i$ with: the root of $\xi_d(k)$ if $f_{>2}(i)=0$; the topmost vertex in the $f_{>2}(i)^{\text{th}}$ vertebra of $\xi_d(k)$ if $1 \leq f_{>2}(i) \leq \left\lfloor \frac{k}{2}\right\rfloor$; and the topmost vertex of the tail if $f_{>2}(i)=\left\lfloor \frac{k}{2}\right\rfloor+1$. 

Consider any $i$ with $s_i>0$ and $f_*(i) \leq \left\lfloor \frac{k}{2}\right\rfloor-1$. If $\tau_i$ is the left subtree of $v_i$, then the inductive hypothesis and the definition of $f_*$ guarantee that we can embed $\tau_i$ into the copy of $\xi_d(f_*(i))$ that is glued to the left leaf of the $f_*(i)^{\text{th}}$ vertebra.  Contract this glued subtree to a copy of $\tau_i$; then contract the right subtree of this vertebra to a point; then contract the vertebra itself to the edge connecting $v_i$ to $\tau_i$ and one other edge below $v_i$ of the same type as the edge connecting $v_i$ and $v_{i+1}$ in $T'$.  The exact same procedure can be done in the case where $\tau_i$ is the right subtree of $v_i$.  If $s_i=0$ and $i\neq m$, then we contract everything in the $f_*(i)^{\text{th}}$ vertebra except for a single edge of the type that connects $v_i$ and $v_{i+1}$ in $T'$.  Things are even easier for $i>0$ with $s_i=0$ and $f_*(i) \leq \left\lfloor \frac{k}{2}\right\rfloor-1$: in this case, we simply embed the (unique) edge below $v_i$ in the ``center'' crescent of the $i^{\text{th}}$ vertebra (i.e., the crescent whose bottom is the center leaf of the vertebra).  When $i=0$ and $s_0=0$, we embed this edge into the crescent at the top of the spine.

For $d=2$, we finish by contracting the tail of $\xi_2(k)$ to a copy of the tree below $v_m$ in $T'$.  This completes the contraction of $\xi_2(k)$ to $T'$, which, as remarked earlier, can be further contracted to $T$.  Now, we turn to $d>2$.  We will later show that if $f_{>2}(m-1)=\left\lfloor \frac{k}{2}\right\rfloor$, then $s_{m-1} \leq \left\lfloor \frac{k+1}{4} \right\rfloor$, so we can embed $\tau_{m-1}$ at the level of the $\left\lfloor \frac{k}{2}\right\rfloor ^{\text{th}}$ $d$-vertebra as in the previous paragraph.  And then we contract the tail of $\xi_d(k)$ to a copy of the tree below $v_m$ in $T'$, which completes the contraction of $\xi_d(k)$ to $T'$.

The next order of business is showing that $f_2$ and $f_{>2}$ have the desired properties for the embedding described above. We first show that $f_2(m-1)\leq \left\lfloor \frac{k}{2}\right\rfloor-1$ and $f_{>2}(m-1) \leq \left\lfloor \frac{k}{2}\right\rfloor$.  Both functions are strictly increasing on $i \leq m-1$, so this will also prove that they are injective. Easy induction on $r$ shows that $$f_*(r) \leq \sum_{i=0}^r \max \{1,s_i\},$$ with equality exactly when $s_0 \geq 1$ and $s_i\leq 1$ for all $1 \leq i \leq r$.  In particular,
\begin{equation}\label{eq:f-inequality}
f_*(m-2) \leq \sum_{i=0}^{m-2} \max \{1,s_i\}.
\end{equation}
At the same time, recall that $\max \{1,s_i\}$ is controlled by the number of edges of $T$ (non-red edges of $T'$) that ``peel away'' at the vertex $v_i$ (compare with \eqref{eq:condition}), so the condition for the termination of the sequence  $(T_i,v_i)$ implies that
\begin{equation}\label{eq:big-inequality}
\sum_{i=0}^{m-2} \max \{1,s_i\} \leq \textstyle{\left\lfloor \frac{k}{2} \right\rfloor}-1.
\end{equation}
This immediately implies the claim about $f_{>2}(m-1)$.

Next, we can show that the inequalities \eqref{eq:f-inequality} and \eqref{eq:big-inequality} cannot both be tight for $d=2$; this will imply that $f_2(m-2)\leq \left\lfloor \frac{k}{2} \right\rfloor-2$, and the inequality $f_2(m-1)\leq \left\lfloor \frac{k}{2} \right\rfloor-1$ will immediately follow from the definition of $f_2$. To see that this improvement is in fact achieved, note that the first condition in equation~\eqref{eq:condition} (which implies an improvement to \eqref{eq:big-inequality}) always occurs somewhere unless $T$ consists of a path on $\left\lfloor \frac{k}{2} \right\rfloor+1$ vertices with a tree on $\left\lceil \frac{k}{2} \right\rceil$ vertices glued to the bottom.  But in this exceptional case, we have $s_0=0$, so inequality~\eqref{eq:f-inequality} is not tight. This completes the proof for $d=2$. 

We still need to check that if $d>2$ and $f_{>2}(m-1)=\left\lfloor \frac{k}{2} \right\rfloor$, then $s_{m-1} \leq \left\lfloor \frac{k+1}{4} \right\rfloor$. Suppose $f_{>2}(m-1)=\left\lfloor \frac{k}{2} \right\rfloor$, so that the inequalities \eqref{eq:f-inequality} and \eqref{eq:big-inequality} are equalities. This means that $f_{>2}(m-2)=\left\lfloor \frac{k}{2} \right\rfloor-1$.  Because \eqref{eq:f-inequality} is an equality, the vertex $v_{m-1}$ has exactly $\left\lceil \frac{k}{2} \right\rceil$ vertices (including $v_{m-1}$) below it in $T_{m-1}$.  If here $v_{m-1}$ has only a single nonempty subtree, then $s_{m-1}=0$ and we are done.  Otherwise, $v_{m-1}$ has at least two nonempty subtrees.  The (weakly) smaller of the rightmost and leftmost of these subtrees must have at most $\left\lfloor \frac{1}{2} \left\lceil \frac{k}{2} \right\rceil \right\rfloor=\left\lfloor \frac{k+1}{4} \right\rfloor$ vertices, so we conclude that $s_{m-1} \leq \left\lfloor \frac{k+1}{4} \right\rfloor$, as desired.
\end{proof}

We remark that in the $d=2$ case, ad hoc arguments show that this construction is in fact optimal for $k \leq 5$; however, small refinements are possible for sufficiently large $k$.

Now that we have shown that the $\xi_d(k)$'s are in fact noncontiguous $k$-universal $d$-ary plane trees, we focus on their sizes.  Let $M_d(k)$ denote the number of vertices in $\xi_d(k)$.  The following proposition, whose proof we omit, follows from counting the various parts of $\xi_d(k)$ as described in the recursive construction above.  Let $\delta_{x,y}$ denote the Kronecker delta, which has the value $1$ when $x=y$ and the value $0$ otherwise.  Note that the $-1$'s below account for ``overlap'' vertices that are contained in multiple parts.  

\begin{proposition}\label{prop:colorful-size}

For fixed $d$, the sequence $M_d(k)$ has the initial conditions \[M_d(1)=1,\quad M_d(2)=d+1,\quad M_d(3)=2d+1,\] and for $k \geq 4$, it obeys the recurrence
\begin{align*}
M_d(k) &=(d+1)+\left( \textstyle{\left\lfloor \frac{k}{2} \right\rfloor}-\delta_{d,2} \right)(3d-2)+2\sum_{i=1}^{\left\lfloor \frac{k}{2} \right\rfloor-2}(M_d(i)-1) +M_d\left(\textstyle{\left\lfloor \frac{k}{2} \right\rfloor}-1\right)-1\\
 &+M_d\left(\textstyle{\left\lceil \frac{k}{2} \right\rceil}-1\right)-1+ M_d\left(\textstyle{\left\lceil \frac{k}{2} \right\rceil}\right)-1+2(1-\delta_{d,2})\left( M_d\left(\textstyle{ \left\lfloor \frac{k+1}{4} \right\rfloor} \right)-1\right).\end{align*}
\end{proposition}

We can use Proposition \ref{prop:colorful-size} and an argument similar to one of the proofs in \cite{Goldberg} to obtain asymptotics for $M_d(k)$.

\begin{corollary}\label{cor:colorful}
For fixed $d \geq 2$, we have
$$N_{d\ary}^{\non}(k)\leq M_d(k)=k^{\frac{1}{2}\log_2(k)(1+o(1))}.$$
\end{corollary}

\begin{proof}
Fix $d$.  It will be convenient to work with natural logarithms, so note that $k^{\frac{1}{2}\log_2(k)}=\exp\left(\frac{1}{2\log 2}\log^2 k\right)$. We first prove that there is some constant $C$ (depending on $d$) such that $M_d(k)<C\exp\left( \frac{1}{2 \log 2} \log^2 k \right)$ for all $k$.  We proceed by induction on $k$, where making $C$ large deals with any base cases.  It is obvious (by construction) that $M_d(k)=|\xi_d(k)|$ is monotonically increasing in $k$.  We compute (for sufficiently large $k$):
\begin{align*}
    M_d(k) &=M_d(k-2)+(3d-2)+2\left(M_d\left( \textstyle{\left\lfloor \frac{k}{2} \right\rfloor}-2 \right)-1\right)+M_d\left( \textstyle{\left\lfloor \frac{k}{2} \right\rfloor}-1 \right)\\
     &\quad -M_d\left( \textstyle{\left\lfloor \frac{k}{2} \right\rfloor}-2 \right)+M_d\left( \textstyle{\left\lceil \frac{k}{2} \right\rceil} \right)-M_d\left( \textstyle{\left\lceil \frac{k}{2} \right\rceil}-2 \right)+2(1-\delta_{d,2})\left(M_d\left( \textstyle{\left\lfloor \frac{k+1}{4} \right\rfloor} \right)-M_d\left( \textstyle{\left\lfloor \frac{k-1}{4} \right\rfloor} \right)\right)\\
      &\leq M_d(k-2)+3d-2+M_d\left( \textstyle{\left\lfloor \frac{k}{2} \right\rfloor}-1 \right)+M_d\left( \textstyle{\left\lceil\frac{k}{2} \right\rceil} \right)+2M_d\left( \textstyle{\left\lfloor \frac{k+1}{4} \right\rfloor} \right)\\
    &< C\exp\left( \frac{1}{2 \log 2} \log^2 (k-2) \right)+5C\exp\left( \frac{1}{2 \log 2} \log^2 \left( \frac{k+1}{2} \right) \right)\\
    &=C\exp\left( \frac{1}{2 \log 2} \log^2 (k-2) \right) \left( 1+5\exp \left(  \frac{1}{2\log 2} \log\left( \frac{(k+1)(k-2)}{2} \right)\log \left(\frac{k+1}{2(k-2)} \right) \right) \right)\\
    &=C\exp\left( \frac{1}{2 \log 2} \log^2 (k-2) \right) \left(1+\frac{5}{k}+o\left(\frac{1}{k} \right) \right).
\end{align*}
At the same time, this expression is certainly smaller than
\begin{align*}
    C\exp\left( \frac{1}{2 \log 2} \log^2 k \right) &=C\exp\left( \frac{1}{2 \log 2} \log^2 (k-2) \right)\left(1+\frac{2\log k}{k \log 2}+o\left(\frac{\log k}{k} \right) \right)
\end{align*}
for sufficiently large $k$, which establishes the first claim.

Second, we show that for any $\gamma <\frac{1}{2 \log 2}$, there exists a constant $c>0$ (depending on $\gamma$) such that $N_d(k)>c\exp(\gamma \log^2 k)$ for all $k$.  As above, we proceed by induction on $k$, where making $c$ small deals with the base cases.  This time, we compute:
\begin{align*}
    M_d(k &)>M_d(k-2)+2M_d\left( \frac{k-5}{2} \right)\\
    &>c\exp \left( \gamma \log^2 (k-2) \right)\left(1+2\exp \left( \gamma \log\left( \frac{(k-5)(k-2)}{2} \right) \log \left( \frac{k-5}{2(k-2)} \right) \right) \right)\\
    &=c\exp \left( \gamma \log^2 (k-2) \right)\left(1+\frac{2}{k^{2\gamma \log 2}}+o\left(\frac{1}{k^{2\gamma \log 2}} \right)  \right).
\end{align*}
Since $2\gamma \log 2<1$, this expression is larger than
$$c\exp \left( \gamma \log^2 k \right)=c\exp \left( \gamma \log^2 (k-2) \right)\left(1+\frac{4\gamma \log k}{k}+o\left(\frac{\log k}{k} \right) \right)$$
for sufficiently large $k$, as desired.  The two claims together imply the result.
\end{proof}

\section{$[d]$-trees}\label{sec:0,1,...,d}

\subsection{Contiguous containment}\label{subsec:0,1,...,d-contiguous}

\subsubsection{Lower bounds}
We can obtain a lower bound for $N^{\text{con}}_{[d]}(k)$ by modifying the argument in the proof of the first part of  Theorem~\ref{thm:d-ary-contiguous}.  In particular, we apply this argument to a slightly different family of trees that are difficult to contain.

\begin{theorem}\label{thm:0,...,d-contiguous-lower}
For $d\geq 2$ and $k \geq 2$, we have \[N_{[d]}^{\con}(k)\geq \left(k-1-d\textstyle{\left\lfloor \frac{k-2}{d} \right\rfloor}\right)d^{\left\lfloor \frac{k-2}{d} \right\rfloor}+\textstyle{\left\lfloor \frac{k-2}{d} \right\rfloor}+1\geq d^{\frac{k-2}{d}}.\]
\end{theorem}

\begin{proof}
Let $\bf T$ be a contiguous $k$-universal $[d]$-tree.  Consider the following procedure for building $[d]$-trees with $k$ vertices. Start with a $[d]$-tree that is a path on $\left\lfloor\frac{k-2}{d}\right\rfloor+1$ vertices, and color the edges of this path red. Add $d-1$ additional children to each nonleaf vertex of this path. When adding these additional children to a nonleaf vertex $v$, we choose freely how many children to place on the left of the red edge coming down from $v$, then we place the remaining children on the right of this red edge. Finally, place $k-1-d\left\lfloor \frac{k-2}{d} \right\rfloor$ vertices as children of the leaf of the original path. This forms a $[d]$-tree with $k$ vertices, and there are $d^{\left\lfloor \frac{k-2}{d} \right\rfloor}$ trees that can be built in this fashion. Each of these trees has $k-1-d\left\lfloor \frac{k-2}{d} \right\rfloor$ vertices at depth $\left\lfloor\frac{k-2}{d}\right\rfloor+1$, an easy modification of the argument in Theorem~\ref{thm:d-ary-contiguous} shows that these vertices must correspond to pairwise distinct vertices in $\bf T$ of depth at least $\left\lfloor\frac{k-2}{d}\right\rfloor+1$.  Thus, $\bf T$ has at least $$\left(k-1-d\textstyle{\left\lfloor \frac{k-2}{d} \right\rfloor}\right)d^{\left\lfloor \frac{k-2}{d} \right\rfloor}$$ vertices at depth at least $\left\lfloor\frac{k-2}{d}\right\rfloor+1$. The term $\left\lfloor\frac{k-2}{d}\right\rfloor+1$ in the statement of the theorem accounts for the fact that $\bf T$ must also have at least one vertex at each depth $0,1,\ldots, \left\lfloor\frac{k-2}{d}\right\rfloor$.
\end{proof}

\subsubsection{Upper bounds}

As mentioned in the introduction, the quantity $N^{\text{con}}_{d\ary}(k)=d^{k-1}+k-1$ is an upper bound for $N^{\text{con}}_{[d]}(k)$.  We can dramatically improve the base of the exponential by describing an explicit construction for a family $\{\Lambda_d(k)\}$ of contiguous $k$-universal $[d]$-trees.

The construction of $\Lambda_d(k)$ is recursive.  We first let $\Lambda_d(1)$ be a single vertex.  For $2 \leq k \leq d$, we construct $\Lambda_d(k)$ by attaching the subtrees $\Lambda_d(1), \Lambda_d(2), \ldots, \Lambda_d\left(\left\lfloor\frac{k}{2}\right\rfloor-1\right)$, then $\Lambda_d(k-1)$, then $\Lambda_d\left(\left\lceil\frac{k}{2}\right\rceil-1\right), \Lambda_d\left(\left\lceil\frac{k}{2}\right\rceil-2\right), , \ldots, \Lambda_d(1)$ to the root, in that order from left to right.  For $k>d$, we construct $\Lambda_d(k)$ by attaching the subtrees $\Lambda_d(k-d), \Lambda_d(k-d+1), \ldots, \Lambda_d\left(k-\left\lfloor\frac{d}{2}\right\rfloor-2\right)$, then $\Lambda_d(k-1)$, then $\Lambda_d\left(k-\left\lceil\frac{d}{2}\right\rceil-1\right), \Lambda_d\left(k-\left\lceil\frac{d}{2}\right\rceil-2\right), \ldots, \Lambda_d(k-d)$ to the root, in that order from left to right.  The proof that these trees are in fact $k$-supertrees is similar in spirit to the proof of Theorem~\ref{thm:colorful-construction}.

\begin{theorem}\label{thm:0,...,d-upper-bound}
Let $d\geq 2$ and $k\geq 1$ be integers. For every $k$-vertex $[d]$-tree $T$, there is a contiguous embedding $T^*$ of $T$ in $\Lambda_d(k)$ such that the root of $T^*$ coincides with the root of $\Lambda_d(k)$. 
\end{theorem}
\begin{proof}
We proceed by strong induction on $k$, where the base case $k=1$ is trivial.  For the induction step, we begin with the case in which $2 \leq k \leq d$.  Let $T$ be a $[d]$-tree on $k$ vertices.  Let $T_1, \ldots, T_{\ell}$ be the subtrees of the root of $T$, from left to right. Of these $\ell$ trees, let $T_m$ be one with the most vertices.   We embed $T_1, \ldots, T_{\ell}$ into the subtrees of the root of $\Lambda_d(k)$ in a greedy way, with a preference for subtrees farther to the left.  We consider two cases based on the size of $T_m$.  If $|T_m|\geq\left\lceil \frac{k}{2} \right\rceil$, then by the induction hypothesis, we can embed $T_m$ in the subtree of the root of $\Lambda_d(k)$ that is isomorphic to $\Lambda_d(k-1)$ (with the roots coinciding). The remaining subtrees, which contain at most $\left\lfloor \frac{k}{2}\right\rfloor-1$ vertices in total, can easily be embedded in the smaller subtrees of the root of $\Lambda_d(k)$. Otherwise, $|T_m|\leq\left\lceil \frac{k}{2} \right\rceil-1 \leq \left\lfloor \frac{k}{2} \right\rfloor$.  In this case, we let $f(1)=|T_1|$ and, for $2\leq i\leq \ell$, let $f(i)=\max\{1+f(i-1),|T_i|\}$.  Note that $f$ is strictly increasing.  We have $$f(r) \leq \sum_{i=1}^r |T_i|\leq k-1-(\ell-r)$$ for each $r$, where the second inequality uses the fact that $|T_i|\geq 1$ for all $i$.  In particular, $f(\ell) \leq k-1$.  We now claim that we can embed each $T_r$ in the $f(r)^{\text{th}}$ subtree of the root of $\Lambda_d(k)$ (with the roots coinciding).  This is certainly possible when $f(r) \leq \left\lfloor \frac{k}{2} \right\rfloor$ by the induction hypothesis because $f(r)\geq |T_r|$. It is also possible when $f(r)> \left\lfloor \frac{k}{2} \right\rfloor$ because $f(r)\leq k-1-(\ell-r)$.  (These two statements also use the fact that $|T_m|\leq \left\lfloor \frac{k}{2} \right\rfloor$.) This completes the argument when $2 \leq k \leq d$.

We now assume $k>d$.  Let $T$ be a $[d]$-tree on $k$ vertices, and let $T_1, \ldots, T_{\ell}$ be the subtrees of its root, from left to right. Again, we have $|T_1|+\cdots+|T_{\ell}|=k-1$.  As above, it is easy to dispense with the case in which some $|T_i| \geq k-\left\lceil \frac{d}{2} \right\rceil$, so we restrict our attention to the case where $|T_i| \leq k-\left\lceil \frac{d}{2} \right\rceil-1\leq k-\left\lfloor \frac{d}{2} \right\rfloor-1$ for all $i$.  Let $g(1)=\max\{1,|T_1|-(k-d)+1\}$, and for $2\leq i\leq\ell$, let $g(i)=\max\{1+g(i-1),|T_i|-(k-d)+1\}$.  Note that $g$ is strictly increasing and
$$g(r)\leq\sum_{i=1}^r \max\{1,|T_i|-(k-d)+1\}$$ for every $r$.  In particular, if we let $h_s$ denote the number of trees in the set $T_1,\ldots,T_\ell$ with exactly $s$ vertices, then for $r=\ell$ we get 
\[g(\ell)\leq\sum_{i=1}^\ell\max\{1,|T_i|-(k-d)+1\}=\sum_{s=1}^{k-d}h_s+\sum_{s\geq k-d+1}(s-(k-d)+1)h_s\]
\[=\sum_{s\geq 1}sh_s-\sum_{s=1}^{k-d}(s-1)h_s-\sum_{s\geq k-d+1}(k-d-1)h_s=k-1-\sum_{s=1}^{k-d}(s-1)h_s-\sum_{s\geq k-d+1}(k-d-1)h_s.\] If $h_s\geq 1$ for some $s\geq k-d+1$, then this shows that $g(\ell)\leq d$. Otherwise, we have \[g(\ell)\leq k-1-\sum_{s=1}^{k-d}(s-1)h_s=k-1-\sum_{s\geq 1}(s-1)h_s=k-1-\sum_{s\geq 1}sh_s+\sum_{s\geq 1}h_s=k-1-(k-1)+\ell\leq d.\] In either case, $g(\ell)\leq d$. Since $g$ is strictly increasing, $g(r)\leq d-(\ell-r)$ for all $r\in\{1,\ldots,\ell\}$. 

We now claim that we can embed each $T_r$ in the $g(r)^{\text{th}}$ subtree of the root of $\Lambda_d(k)$ (with the roots coinciding).  As above, this is possible whenever $g(r) \leq \left\lceil \frac{d}{2} \right\rceil$ because $g(r) \geq |T_r|-(k-d)+1$.  It is also possible when $g(r)>\left\lceil \frac{d}{2} \right\rceil$ because $g(r)\leq d-(\ell-r)$. 
\end{proof}

We can also describe the sizes of these $k$-universal $[d]$-trees.  Let $L_d(k)$ denote the number of vertices in $\Lambda_d(k)$.  The following enumeration follows directly from the definition of $\Lambda_d(k)$.
\begin{proposition}\label{prop:0,...,d-contiguous-upper}
For fixed $d \geq 2$, the sequence $L_d(k)$ has the starting value $L_d(1)=1$.  For $2 \leq k \leq d$, we have the recurrence
$$L_d(k)=1+L_d(k-1)+\sum_{i=1}^{\left\lfloor \frac{k}{2} \right\rfloor-1} L_d(i)+\sum_{i=1}^{\left\lceil \frac{k}{2} \right\rceil-1} L_d(i).$$  For $k>d$, we have the recurrence
$$L_d(k)=1+L_d(k-1)+\sum_{i=k-d}^{k-\left\lfloor \frac{d}{2} \right\rfloor-2} L_d(i)+\sum_{i=k-d}^{k-\left\lceil \frac{d}{2} \right\rceil-1} L_d(i).$$
\end{proposition}

\begin{corollary}\label{cor:rhoestimate}
Let \[p_d(x)=1-x-\sum_{i=\left\lfloor\frac{d}{2}\right\rfloor+2}^dx^i-\sum_{i=\left\lceil\frac{d}{2}\right\rceil+1}^dx^i.\] Let $\rho_d=1/x_d$, where $x_d$ is the smallest positive real root of $p_d(x)$. For each fixed $d$, we have \[N_{[d]}^{\con}(k)\leq L_d(k)=(\rho_d+o(1))^k.\] Furthermore, as $d\to\infty$, we have \[\rho_d=1+\frac{4\log d}{d}-\frac{4\log\log d}{d}+o\left(\frac{\log\log d}{d}\right).\]
\end{corollary}
\begin{proof}
The inequality $N_{[d]}^{\con}(k)\leq L_d(k)$ follows immediately from Theorem \ref{thm:0,...,d-upper-bound}. To see that $L_d(k)=(\rho_d+o(1))^k$, we let $G_d(x)=\sum_{k\geq 1}L_d(k)x^k$. Using the recurrence in Proposition \ref{prop:0,...,d-contiguous-upper}, it is straightforward to check that $G_d(x)$ is a rational function with denominator $p_d(x)$. Since $G_d(x)$ has nonnegative coefficients, it follows from Pringsheim's theorem \cite[Chapter IV]{FlajoletBook} that $x_d$ is the radius of convergence of $G_d(x)$. This means that $L_d(k)=(\rho_d+o(1))^k$. 

To prove the last statement of the corollary, we consider only the case in which $d$ is odd; the argument is similar when $d$ is even. All asymptotics are as $d\to\infty$. First, note that $$p_d(x)=1-x-2(x^{\frac{d+1}{2}+1}+x^{\frac{d+1}{2}+2}+\cdots+x^d)=1-x-2\frac{x^{c+1}-x^{2c}}{1-x},$$ where $c=\frac{d+1}{2}$.  We have \[(1-x_d)^2-2x_d^{c+1}+2x_d^{2c}=0.\] The additional substitution $x_d=1-\frac{\varepsilon_d}{c}$ (where clearly $\varepsilon_d>0$ since $L_d(k)$ is growing with $k$) gives
$$\left(\frac{\varepsilon_d}{c}\right)^2-2\left(1-\frac{\varepsilon_d}{c}\right)^{c+1}+2\left(1-\frac{\varepsilon_d}{c}\right)^{2c}=0.$$  One can show that $x_d\to 1$, so $\frac{\varepsilon_d}{c}\to 0$.
Now, $2\left(1-\frac{\varepsilon_d}{c}\right)^{c+1}=2e^{-\varepsilon_d}+o(e^{-\varepsilon_d})$ and $2\left(1-\frac{\varepsilon_d}{c}\right)^{2c}=o(e^{-\varepsilon_d})$. This means that \[\left(\frac{\varepsilon_d}{c}\right)^2=2e^{-\varepsilon_d}(1+o(1))\] and \[\frac{\varepsilon_d}{c}=\sqrt{2}e^{-\varepsilon_d/2}(1+o(1)).\] Rearranging, we find that \[\frac{\varepsilon_d}{2}e^{\varepsilon_d/2}=\frac{c}{\sqrt 2}(1+o(1)).\] Therefore, \[\varepsilon_d=2W\left(\frac{c}{\sqrt 2}(1+o(1))\right)=2\log c-2\log\log c+o(\log\log c),\] where $W$ is the Lambert $W$ function. The desired result follows.
\end{proof}

\subsection{Noncontiguous containment}\label{subsec:0,1,...,d-noncontiguous}

The following theorem relates noncontiguous $k$-universal $[d]$-trees with noncontiguous $k$-universal $d$-ary plane trees. In particular, it shows that the minimum sizes of these trees differ by at most a constant factor. 

\begin{theorem}\label{thm:0,...,d-noncontiguous}
For all integers $d \geq 2$ and $k \geq 1$, we have
$$N_{[d]}^{\non}(k) \leq N_{d\ary}^{\non}(k) \leq d (N_{[d]}^{\non}(k)-1)+1.$$
\end{theorem}

\begin{proof}
The first inequality is straightforward because if we are given a noncontiguous $k$-universal $d$-ary plane tree ${\bf T}'$, then we can obtain a noncontiguous $k$-universal $[d]$-tree $\bf T$ by ``forgetting'' the exact types of all of the edges in ${\bf T}'$. In other words, we interpret ${\bf T}'$ as a $[d]$-tree.

For the second inequality, suppose $\bf T$ is a noncontiguous $k$-universal $[d]$-tree on $n$ vertices.  We obtain a noncontiguous $k$-universal $d$-ary plane tree ${\bf T}'$ on $d(n-1)+1$ vertices by doing the following for each edge $e$ in $\bf T$.  Among all edges with the same parent vertex as $e$, suppose $e$ is the $i^{\text{th}}$ from the left.   We replace the edge $e$ (along with its endpoints) with a $d$-ary plane tree path on $d$ edges whose topmost edge has type $i$ and whose remaining edges have types $1$ through $d$, skipping $i$.  Note that $|{\bf T}'|= d(|{\bf T}|-1)+1$ since each of the $|{\bf T}|-1$ edges in $\bf T$ has become $d$ edges in ${\bf T}'$.

We claim that ${\bf T}'$ is in fact a noncontiguous $k$-universal $d$-ary plane tree.  Let $T'$ be a $d$-ary plane tree on $k$ vertices, and let $T$ be the corresponding $[d]$-tree that is obtained by forgetting the types of the edges in $T'$.  By hypothesis, $\bf T$ noncontiguously contains $T$.  For each edge $e'$ in $T'$, let $e$ be the corresponding edge in $T$.  Let $\bf e$ be the edge in $\bf T$ that corresponds to $e$ in the noncontiguous embedding of $T$ in $\bf T$.  Recall that $\bf e$ becomes $d$ edges, one of each type, in ${\bf T}'$; let ${\bf e}'$ be the edge among these with the same type as $e'$.  Color every such edge ${\bf e}'$ blue, and color all other edges of ${\bf T}'$ red.  It is clear that if we contract away all of the red edges, the blue edges will form a copy of $T'$, so it remains only to show that a sequence of legal contractions exists.

We begin by contracting every red edge whose bottom vertex is a leaf.  We continue this process until every leaf is incident to a blue edge.  We contract the remaining red edges, starting with those at the greatest depth (i.e., farthest from the root) and working our way upwards.  So we can always assume that all edges of greater depth than our red edges of interest are blue.  If the top vertex of a red edge has no other nonempty subtree, then we can legally contract that red edge.  We are now left with the case where the top vertex $v$ of our red edge $r$ has multiple children.  Consider the nonempty subtrees of $v$, from left to right.  If $r$ is not adjacent to another red edge, then we can contract $r$ immediately.  Otherwise, there are consecutive red edges $r_1, \ldots, r_s$ ($s \geq 2$).  We will show that there is some $r_i$ that we can legally contract; we will then be able to sequentially contract the remaining edges by induction.


Let each $r_i$ have edge type $a_i$.  Let $b_i$ denote the minimum type of a (necessarily blue) edge directly below $r_i$, and let $c_i$ denote the maximum type of a (necessarily blue) edge directly below $r_i$.  It follows from our construction that $$a_1<\cdots<a_s \quad \text{and} \quad b_1\leq c_1<b_2 \leq c_2\cdots<b_s\leq c_s.$$  The condition for being able to legally contract $r_1$ is $c_1<a_2$, and the condition for being able to legally contract $r_s$ is $a_{s-1}<b_s$.  For $2 \leq i \leq s-1$, the conditions for being able to contract $r_i$ are $a_{i-1}<b_i$ and $c_i<a_{i+1}$.  Assume (for contradiction) that we cannot legally contract any  of the edges $r_1,\ldots,r_s$.  Since we cannot contract $r_1$, we must have $c_1 \geq a_2$.  Since we cannot contract $r_2$, we must have either $a_1 \geq b_2$ or $c_2 \geq a_3$.  In the first case, we get
$$b_2\leq a_1<a_2\leq c_1<b_2,$$
which is a contradiction, so we conclude that $c_2 \geq a_3$.  Similarly, since we cannot contract $a_3$, we have either $a_2 \geq b_3$ or $c_3 \geq a_4$, and the first possibility yields a contradiction in the same way.  Continuing this line of reasoning, we arrive at $c_{s-1}\geq a_s$.  Then
$$a_{s-1}<a_s\leq c_{s-1}<b_s$$
tells us that we can legally contract $r_s$, so we are done.  This demonstrates that we can legally contract all of the red edges. 
\end{proof}

Now that we have established this connection between $d$-ary plane trees and $[d]$-trees, we revisit the construction of the trees $\xi_d(k)$ from Section~\ref{subsec:d-ary-noncontiguous}.  Because $[d]$-trees have more ``flexibility'' than $d$-ary plane trees, we can use a slightly better (and simpler!) construction to beat the first inequality in Theorem~\ref{thm:0,...,d-noncontiguous}.  We call these new noncontiguous $k$-universal $[d]$-trees $\Xi_d(k)$.

First, we use the path on $2$ vertices instead of the $d$-crescent.  If $d>2$, we define the \textit{modified $d$-vertebra} to be the $[d]$-tree on $4$ vertices in which the root has $3$ children; when $d=2$, the \emph{modified $2$-vertebra} is the $[2]$-tree with $5$ vertices in which the root has $2$ children and the left child of the root has $2$ children.  As in the case of the $d$-vertebra, we can identify the left, middle, and right children of the modified vertebra in the obvious way.  We then construct the $m^{\text{th}}$ spine exactly as in Section~\ref{subsec:d-ary-noncontiguous}.

Our recursive definition of the families $\Xi_d(k)$ resembles the presentation of Section~\ref{subsec:d-ary-noncontiguous}.  We begin with the following base cases:
\begin{itemize}
    \item Let $\Xi_d(1)$ consist of a single vertex.
    \item Let $\Xi_d(2)$ be the path on $2$ vertices (scil., the analogue of the crescent).
    \item Obtain $\Xi_d(3)$ from the path on $2$ vertices by giving the bottom vertex $2$ children.
\end{itemize}
The construction for larger $k$ is recursive and differs for $d=2$ and $d>2$.  If $d=2$, then for $k \geq 4$, we obtain $\Xi_2(k)$ from the $\left(\left\lfloor \frac{k}{2}\right\rfloor-1\right)^{\text{th}}$ $2$-spine as follows:
\begin{enumerate}
    \item For each $1 \leq i \leq \left\lfloor \frac{k}{2}\right\rfloor-2$, glue a copy of $\Xi_2(i)$ to each of the left and right leaves of the $i^{\text{th}}$ modified $2$-vertebra.
    \item Glue a copy of $\Xi_2(\left\lfloor \frac{k}{2}\right\rfloor-1)$ to the right leaf of the $\left(\left\lfloor \frac{k}{2}\right\rfloor-1\right)^{\text{th}}$ (i.e., lowest) modified $2$-vertebra.
    \item Glue a copy of $\Xi_2(\left\lceil\frac{k}{2}\right\rceil-1)$ to the left leaf of the $\left(\left\lfloor \frac{k}{2}\right\rfloor-1\right)^{\text{th}}$ modified $2$-vertebra.
    \item Glue a copy of $\Xi_2(\left\lceil\frac{k}{2}\right\rceil)$ to the center leaf of the $\left(\left\lfloor \frac{k}{2}\right\rfloor-1\right)^{\text{th}}$ modified $2$-vertebra.
\end{enumerate}

If $d>2$, then for $k \geq 4$, we obtain $\Xi_d(k)$ from the $\left(\left\lfloor \frac{k}{2}\right\rfloor-1\right)^{\text{th}}$ $d$-spine as follows:
\begin{enumerate}
    \item For each $1 \leq i \leq \left\lfloor \frac{k}{2}\right\rfloor-2$, glue a copy of $\Xi_d(i)$ to each of the left and right leaves of the $i^{\text{th}}$ modified $d$-vertebra.
    \item Glue a copy of $\Xi_d(\left\lfloor \frac{k}{2}\right\rfloor-1)$ to the right leaf of the $\left(\left\lfloor \frac{k}{2}\right\rfloor-1\right)^{\text{th}}$ (i.e., second-lowest) modified $d$-vertebra.
    \item Glue a copy of $\Xi_d(\left\lceil\frac{k}{2}\right\rceil-1)$ to the left leaf of the $\left(\left\lfloor \frac{k}{2}\right\rfloor-1\right)^{\text{th}}$ modified $d$-vertebra.
    \item Glue a copy of $\Xi_d(\left\lceil\frac{k}{2}\right\rceil)$ to the center leaf of the $\left\lfloor \frac{k}{2}\right\rfloor^{\text{th}}$ (i.e., lowest) modified $d$-vertebra.
    \item Glue a copy of $\Xi_d\left( \left\lfloor \frac{k+1}{4} \right\rfloor \right)$ to each of the left and right leaves of the $\left\lfloor \frac{k}{2}\right\rfloor^{\text{th}}$ $d$-vertebra.
\end{enumerate}
For $k\geq 4$, we still say that the \textit{tail} of $\Xi_d(k)$ is the copy of $\Xi_d(\left\lceil \frac{k}{2} \right\rceil)$ that is glued to the center leaf of the bottom of the spine in step (4).  We remark that the trees $\Xi_3(k), \Xi_4(k),\ldots$ are all identical.

We omit the proof of the following theorem because it is identical to the proof of Theorem~\ref{thm:colorful-construction}.
\begin{theorem}\label{thm:modified-colorful-construction}
For any integers $d \geq 2$ and $k \geq 1$, the tree $\Xi_d(k)$ noncontiguously contains all $[d]$-trees with $k$ vertices.
\end{theorem}
Also as before, simple counting gives a recursive formula for the number of vertices in $\Xi_d(k)$, which we denote $M'_d(k)$.
\begin{proposition}\label{prop:modified-colorful-size}
For fixed $d$, the sequence $M'_d(k)$ has the initial conditions
\[M_d'(1)=1,\quad M_d'(2)=2,\quad M_d'(3)=4.
\]
For $k \geq 4$, it obeys the recurrence
\begin{align*}
M'_d(k) &=2+(3+\delta_{d,2})\left( \textstyle{\left\lfloor \frac{k}{2} \right\rfloor}-\delta_{d,2} \right)+2\sum_{i=1}^{\left\lfloor \frac{k}{2} \right\rfloor-2} (M'_d(i)-1) +M'_d\left(\textstyle{\left\lfloor \frac{k}{2} \right\rfloor}-1\right)-1\\
 &+M'_d\left(\textstyle{\left\lceil \frac{k}{2} \right\rceil}-1\right)-1+M'_d\left(\textstyle{\left\lceil \frac{k}{2} \right\rceil}\right)-1 +2(1-\delta_{d,2})\left( M_d\left(\textstyle{ \left\lfloor \frac{k+1}{4} \right\rfloor} \right)-1\right).\end{align*}
\end{proposition}
The proof of Corollary~\ref{cor:colorful} carries through to show that $M'_d(k)=k^{\frac{1}{2}\log_2(k)(1+o(1))}$. 

\begin{corollary}\label{cor:modified-colorful}
For fixed $d \geq 2$, we have
$$N_{[d]}^{\non}(k)\leq M'_d(k)=k^{\frac{1}{2}\log_2(k)(1+o(1))}.$$
\end{corollary}

\section{Conclusions}\label{sec:conclusions}

In Section~\ref{sec:d-ary}, we found the exact values of $N_{d\ary}^{\con}(k)$ for all $d\geq 2$ and $k\geq 1$. Furthermore, the lower and upper bounds that we obtained for $N_{[d]}^{\con}(k)$ are relatively close to each other. By contrast, our lower and upper bounds for $N_{d\ary}^{\non}(k)$ and $N_{[d]}^{\non}(k)$ are far apart. This is largely because it is difficult to obtain good lower bounds for the sizes of noncontiguous universal objects, which is also true in the setting of universal permutations. It would be nice to have better methods for producing lower bounds. Of course, we also encourage the interested reader to try improving our upper bounds. 

Theorem \ref{thm:0,...,d-noncontiguous} leads us naturally to ask the following. 

\begin{question}
Fix $d\geq 2$. Does the limit \[\lim_{k \to \infty} \frac{N_{d\ary}^{\non}(k)}{N_{[d]}^{\non}(k)}\] exist, and, if so, what is its value?
\end{question}

Theorem \ref{thm:0,...,d-contiguous-lower} and Corollary \ref{cor:rhoestimate} suggest that $N_{[d]}^{\con}(k)$ has an exponential growth rate. It would be interesting to know its value, beyond the bounds $d^{\frac{1}{d}}$ and $\rho_d$ provided.
\begin{question}
Fix $d\ge 2$. Does the limit
\[\lim_{k\to\infty}N_{[d]}^{\con}(k)^{\frac{1}{k}}\]
exist, and, if so, what is its value?
\end{question}

The articles \cite{CGSCaterpillars, ChungGraham1, ChungGraham2, ChungGraham3, ChungGraham4} investigate universal trees, where the trees under consideration are unrooted and nonplane. In this setting, a tree $\mathcal T$ contains a tree $T$ if $T$ is an induced subgraph of $\mathcal T$. It would likely be interesting to consider analogous questions in a noncontiguous framework. More precisely, say that a tree $\mathcal T$ noncontiguously contains a tree $T$ if it is possible to obtain $T$ by performing a sequence of edge contractions on $\mathcal T$. In this setting, what is the smallest size of a tree that noncontiguously contains all $k$-vertex trees?   

There has also been recent interest in pattern containment/avoidance in labeled rooted trees \cite{Baril, Dotsenko}. It would be interesting to examine universal trees in these contexts, as well.

\section{Acknowledgements}\label{sec:acknowledgements}

The authors would like to thank Joe Gallian for hosting them at the University of Minnesota, Duluth, where much of this research was conducted with partial support from NSF/DMS grant 1659047 and NSA grant H98230-18-1-0010. The authors also would like to thank the anonymous referee for helpful comments. The first author was additionally supported by a Fannie and John Hertz Foundation Fellowship and an NSF Graduate Research Fellowship.

\end{document}